\documentclass[a4paper,10pt]{article}
\usepackage{graphicx}

\usepackage{amsmath} 
\usepackage{amssymb} 
\usepackage{mathrsfs} 
\usepackage{mathabx} 
\usepackage{titlesec} 

\usepackage[shortlabels]{enumitem} 
\newlist{assumptions}{enumerate}{1} 

\usepackage[ruled]{algorithm2e} 

\usepackage[dvipsnames]{xcolor}
\usepackage{hyperref} 
\hypersetup{
	colorlinks=true, 
	linkcolor=black, 
	citecolor=black, 
	urlcolor=blue 
}

\usepackage{geometry} 

\usepackage{amsthm}
\theoremstyle{theorem}
\newtheorem{theorem}{Theorem}[section]
\newtheorem{lemma}[theorem]{Lemma}
\newtheorem{proposition}[theorem]{Proposition}

\theoremstyle{definition}
\newtheorem{definition}[theorem]{Definition}

\newtheorem{remark}[theorem]{Remark}

\begin{document}

\title{Riemannian Interior Point Methods \\ for Constrained Optimization on Manifolds\footnote{An earlier version of this article has been circulated under the title ``Superlinear and Quadratic Convergence of Riemannian Interior Point Methods.''}}

\author{Zhijian Lai \and Akiko Yoshise}

\author{
	Zhijian Lai\thanks{
		University of Tsukuba, 
		Tsukuba, Ibaraki 305-8573, Japan.
		\href{mailto:galvin.lai@outlook.com}{galvin.lai@outlook.com}
	}
	\and Akiko Yoshise\thanks{
		Corresponding author. 
		University of Tsukuba, 
		Tsukuba, Ibaraki 305-8573, Japan.
		\href{mailto:yoshise@sk.tsukuba.ac.jp}{yoshise@sk.tsukuba.ac.jp}
	}
}

\date{Revised \today} 

\maketitle

\begin{abstract}

We extend the classical primal-dual interior point method from the Euclidean setting to the Riemannian one. 
Our method, named the Riemannian interior point method, is for solving Riemannian constrained optimization problems. 
We establish its local superlinear and quadratic convergence under the standard assumptions.
Moreover, we show its global convergence when it is combined with a classical line search.
Our method is a generalization of the classical framework of primal-dual interior point methods for nonlinear nonconvex programming.
Numerical experiments show the stability and efficiency of our method.

\smallskip

\textbf{Communicated by Sándor Zoltán Németh}

\smallskip


\end{abstract}

%

\section{Introduction}

In this paper, we consider the following Riemannian Constrained Optimization Problem (RCOP):
\begin{equation}\tag{RCOP} \label{RCOP} 
	\begin{array}{ll}
		\underset{x \in \mathcal{M}}{\min} & f(x) \\
		\text { s.t. } & h(x)=0, \text{ and }  g(x) \leq 0,\\
	\end{array}
\end{equation}
where $\mathcal{M}$ is a connected, complete $d$-dimensional Riemannian manifold and $f \colon \mathcal{M}\to \mathbb{R}, h\colon\mathcal{M}\to \mathbb{R}^{l}$, and $g\colon\mathcal{M}\to \mathbb{R}^{m}$ are smooth functions.
This problem appears in many applications, for instance, matrix approximation with nonnegative constraints on a fixed-rank manifold \cite{Song2020} and orthogonal nonnegative matrix factorization on the Stiefel manifold \cite{Jiang2022}; for more, see \cite{liu2020simple,Obara2022}.

The body of knowledge on (\ref{RCOP}) without $ h,g$, often called simply \emph{Riemannian optimization}, has grown considerably in the last 20 years. 
The well-known methods in the Euclidean setting, such as steepest descent, Newton, and trust region, have been extended to the Riemannian setting \cite{absil2009optimization,hu2020brief,boumal2022intromanifolds}. 
By contrast, research on (\ref{RCOP}) is still in its infancy. 
The earliest studies go back to ones on the optimality conditions. Yang et al. \cite{yang2014optimality} extended the Karush Kuhn Tucker (KKT) conditions to (\ref{RCOP}). Bergmann and Herzog \cite{bergmann2019intrinsic} considered more Constraint Qualifications (CQs) on manifolds. 
Yamakawa and Sato \cite{yamakawa2022sequential} proposed sequential optimality conditions in the Riemannian case. 
Liu and Boumal \cite{liu2020simple} were the first to develop practical algorithms. 
They extended the augmented Lagrangian method and exact penalty method to (\ref{RCOP}). 
Schiela and Ortiz \cite{Schiela2021} and Obara et al. \cite{Obara2022} proposed the Riemannian sequential quadratic programming method. 
However, to our knowledge, interior point methods have yet to be considered for (\ref{RCOP}).

The advent of interior point methods in the 1980s greatly advanced the field of optimization \cite{wright1997primal,ye1997interior}. By the early 1990s, the success of these methods in linear and quadratic programming ignited interest in using them on nonlinear nonconvex cases \cite{el1996formulation,yamashita1996superlinear}. From the 1990s to the first decade of the 21st century, a large number of interior point methods for nonlinear programming emerged. They proved to be as successful as the linear ones. A subclass known as \emph{primal-dual} interior point methods is the most efficient and practical. As described in \cite{lustig1991computational}, the primal-dual approach to linear programming was introduced in \cite{megiddo1989pathways}: it was first developed as an algorithm in \cite{kojima1989primal} and eventually became standard for the nonlinear case as well \cite{el1996formulation,yamashita1996superlinear}.

\textbf{Contribution.}
In this paper, we extend the primal-dual interior point algorithms from the Euclidean setting, i.e., $\mathcal{M}=\mathbb{R}^{d}$ in (\ref{RCOP}), to the Riemannian setting. We call this extension the \emph{Riemannian Interior Point Method} (RIPM). 
Our contributions are summarized as follows:
\begin{enumerate}
	\item 
	To our knowledge, this is the first study to apply the primal-dual interior point method to the nonconvex constrained optimization problem on Riemannian manifolds. 
	One significant contribution is that we establish many essential foundational concepts for the general interior point method in the Riemannian context, such as the KKT vector field and its covariant derivative. 
	In addition, we build the first framework for the Riemannian version of the interior point method. 
	These contributions will have uses in the future, especially in developing more advanced interior point methods.
	
	\item 
	We give a detailed theoretical analysis to ensure local and global convergence of RIPM. Considering that many practical problems involve minimizing a nonconvex function on Riemannian manifolds, the theoretical counterparts of our method are the early interior point methods for nonlinear nonconvex programming first proposed by El-Bakry et al. \cite{el1996formulation}.
		
	\item 
	Our numerical experiments\footnote{The code is freely available at \url{https://doi.org/10.5281/zenodo.10612799}} demonstrate the great potential of RIPM. The method meets the challenges presented in these experiments with better stability and higher accuracy compared with the current Riemannian exact penalty, augmented Lagrangian, and sequential quadratic programming methods. 
\end{enumerate}

\textbf{Organization.} 
The rest of this paper is organized as follows. 
In Section \ref{sec:Notation}, we review the notation of Riemannian geometry and explain the Riemannian Newton method. 
In Section \ref{sec:Interpretation}, we give a full interpretation of our RIPM and describe a prototype algorithm of RIPM. We also investigate the use of Krylov subspace methods to efficiently solve a condensed form of a perturbed Newton equation. This is particularly important for the numerical implementation of RIPM.
Section \ref{sec:Preliminaries} gives the necessary preliminaries and auxiliary results needed to prove convergence in our subsequent sections.
In Section \ref{sec:local}, we give the proof of local superlinear and quadratic convergence of the prototype algorithm of RIPM. 
Section \ref{sec:global} describes a globally convergent version of RIPM with a classical line search; then, Section \ref{sec:ProofGlobal} proves its global convergence.
Section \ref{sec:ne} is a collection of numerical experiments. 
Section \ref{sec:Conclusion} summarizes our research and presents future work.

\section{Notation}\label{sec:Notation}

\subsection{Riemannian Geometry}

Let us briefly review some concepts from Riemannian geometry, following the notation of \cite{boumal2022intromanifolds}.
$\mathcal{M}$ denotes a finite-dimensional smooth manifold.
Let $p \in \mathcal{M}$ and $T_p \mathcal{M}$ be the tangent space at $p$ with $0_{p}$ being its zero element.
We use a canonical identification $ T_{p} \mathcal{E}\cong \mathcal{E}$ for a vector space $ \mathcal{E} $ and $p \in \mathcal{E}$.
A vector field is a map $V \colon \mathcal{M} \to T \mathcal{M}$ with $V(p) \in T_{p} \mathcal{M}$, where $T \mathcal{M}:=\bigcup_{p \in \mathcal{M}} T_p \mathcal{M}$ is the tangent bundle. 
$\mathfrak{X}(\mathcal{M})$ denotes the set of all smooth vector fields defined on $\mathcal{M}$.
Furthermore, $\mathcal{M}$ is a Riemannian manifold if it is equipped with a Riemannian metric, that is, a choice of inner product $\langle\cdot, \cdot\rangle_p \colon T_p \mathcal{M} \times T_p \mathcal{M} \to \mathbb{R} $ for each tangent space at $p$ on $\mathcal{M}$ such that for all $V, W \in \mathfrak{X}(\mathcal{M})$, the map
\begin{equation}\label{eq:RiemannianMetric}
	p \mapsto\langle V(p), W(p)\rangle_p
\end{equation}
is a smooth function from $\mathcal{M}$ to $\mathbb{R}$.
Riemannian metric induces the norm $\|\xi\|_p:=\sqrt{\langle\xi, \xi\rangle_p}$ for $\xi \in T_p \mathcal{M}$.
We often omit the subscript $p$ if it is clear from the context.
Throughout this paper, we assume that all the manifolds involved are \emph{connected} and \emph{complete}.
Given a curve segment on $\mathcal{M}$, $c:[a, b] \rightarrow \mathcal{M}$, the length of $c$ is defined as 
$
L(c):=\int_a^b \| \dot{c}(t)\|_{c(t)} \mathrm{d} t,
$
where $\dot{c}(t) \in T_{c(t)} \mathcal{M}$ is the velocity vector of $c$ at $t$.
Since $\mathcal{M}$ is connected, there exists a curve segment connecting any pair of points $p,q \in \mathcal{M}$.
Indeed, $\mathcal{M}$ is a metric space under the Riemannian distance $d(p, q):=\inf_{c} L(c) $ where the infimum is taken over all curve segments joining $p $ to $q $ \cite[Thm. 2.55]{lee2018}.
For two manifolds $\mathcal{M}_1, \mathcal{M}_2$ and a smooth map $F \colon \mathcal{M}_1 \to \mathcal{M}_2$, the differential of $F$ at $p \in \mathcal{M}_1$ is a linear operator denoted as $\mathcal{D} F(p) \colon T_p \mathcal{M}_1 \to T_{F(p)} \mathcal{M}_2$.
Let $\mathfrak{F}(\mathcal{M})$ be the set of all smooth \emph{scalar fields} (or say real-valued functions) $f\colon \mathcal{M} \to \mathbb{R}$.
The Riemannian gradient of $f$ at $p$, $\operatorname{grad} f(p)$, is defined as the unique element of $T_p \mathcal{M}$ that satisfies $\langle \xi, \operatorname{grad} f(p)\rangle_p = \mathcal{D} f(p)[\xi]$ for all $\xi \in T_p \mathcal{M}$, where $\mathcal{D} f(p) \colon T_p \mathcal{M} \to T_{f(p)} \mathbb{R} \cong \mathbb{R}$.
Note that for any $f \in \mathfrak{F}(\mathcal{M})$, the gradient vector field $x \mapsto \operatorname{grad} f(x)$ is a smooth vector field on $\mathcal{M}$, i.e., $\operatorname{grad} f \in \mathfrak{X}(\mathcal{M})$.
A retraction $\mathrm{R} \colon T \mathcal{M} \to \mathcal{M}$ is a smooth map such that
$\mathrm{R}_{p}\left(0_{p}\right)=p$ and $\mathcal{D} \mathrm{R}_{p}\left(0_{p}\right) =\operatorname{id}_{T_p \mathcal{M}}$, i.e., the identity map on $T_{p} \mathcal{M}$, where $\mathrm{R}_{p}$ is the restriction of $\mathrm{R}$ to $T_{p} \mathcal{M}$ and $\mathcal{D} \mathrm{R}_{p}\left(0_{p}\right) \colon T_{0_p} (T_p \mathcal{M}) \cong T_p \mathcal{M} \to T_{p} \mathcal{M}$ is the differential of $\mathrm{R}_{p}$ at $0_{p}$.
One theoretically perfect type of retraction is the exponential map, denoted as $\operatorname{Exp}$. Since $\mathcal{M}$ is complete, the exponential map is well-defined on the whole tangent bundle.
Let $\operatorname{Exp}_p\colon T_p \mathcal{M} \rightarrow \mathcal{M}$ be the exponential map at $p$; then $t \mapsto \operatorname{Exp}_p(t \xi)$ is the unique geodesic that passes through $p$ with velocity $\xi \in T_p \mathcal{M}$ when $t=0$.

\subsection{Riemannian Newton Method}\label{subsec:RiemannianNewton}

The Newton method is a powerful tool for finding the zeros of nonlinear functions in the Euclidean setting.
The generalized Newton method has been studied in the Riemannian setting; it aims to find a \emph{singularity} for the vector field $F \in \mathfrak{X}(\mathcal{M})$, i.e., a point $p \in \mathcal{M}$ such that,
\begin{equation}\label{eq:Singularity}
	F(p)=0_{p} \in T_p \mathcal{M}.
\end{equation}
Let $\nabla$ be the Levi-Civita connection on $\mathcal{M}$, i.e., the unique symmetric connection compatible with the Riemannian metric.
The covariant derivative $\nabla F$ assigns each point $p \in \mathcal{M}$ a linear operator $\nabla F(p)\colon T_{p}\mathcal{M}\to T_{p}\mathcal{M}$. 
In particular, the Riemannian Hessian of $f \in \mathfrak{F}(\mathcal{M})$ at $p$ is a self-adjoint operator on $T_{p} \mathcal{M}$, defined as $\operatorname{Hess} f(p):=\nabla (\operatorname{grad}f)(p).$
The Riemannian Newton method for solving (\ref{eq:Singularity}) is performed as Algorithm \ref{algo:StandardNewton}.

\begin{algorithm}
	\caption{Riemannian Newton Method for (\ref{eq:Singularity})}
	\label{algo:StandardNewton}
	\SetAlgoLined
	
	\KwIn{A vector field $F \in \mathfrak{X}(\mathcal{M})$, an initial point $p_0 \in \mathcal{M}$ and a retraction $\mathrm{R}$ on $\mathcal{M}$.}
	\KwOut{Sequence $\{p_k\} \subset \mathcal{M}$ such that $\left\{p_k\right\} \rightarrow p^*$ and $F(p^{*})=0_{p^{*}}$.}	
	
	Set $k \to 0$\;
	
	\While{Stopping criterion not satisfied}{
		1. Obtain $\xi_k \in T_{p_k}\mathcal{M}$ by solving the Newton equation (a linear operator equation on tangent space $T_{p_k}\mathcal{M}$): 
		\begin{equation}\label{eq:StandardNewtonEquation}
			\nabla F(p_k)[\xi_k] = -F(p_k);
		\end{equation}
	
		2. Compute the next point as $p_{k+1} := \mathrm{R}_{p_k}(\xi_k)$\;	
		
		3. $k \to k + 1$\;
	}
\end{algorithm}

\section{Description of Riemannian Interior Point Methods}
\label{sec:Interpretation}

In this section, we will give a comprehensive interpretation of the Riemannian interior point method.
Following common usage in the interior-point literature, big letters denote the associated diagonal matrix, e.g., $Z =\operatorname{diag}(z_{1}, \ldots, z_{m})$ with $z \in \mathbb{R}^{m}$. $e$ denotes the all-ones vector, and $0$ stands for zero vector/matrix with proper dimensions.

\subsection{KKT Vector Field and Its Covariant Derivative}

The Lagrangian function of (\ref{RCOP}) is
\begin{equation*}
	\mathcal{L}(x, y, z):=f(x)+\sum_{i=1}^l y_i h_i(x)+\sum_{i=1}^m z_i g_i(x),
\end{equation*}
where $y \in \mathbb{R}^l$ and $z \in \mathbb{R}^m$ are Lagrange multipliers corresponding to the equality and
inequality constraints, respectively.
With respect to the variable $x$, $\mathcal{L}(\cdot, y, z)$ is a real-valued function defined on $\mathcal{M}$, and its Riemannian gradient is
\begin{equation*}
\operatorname{grad}_{x} \mathcal{L}(x, y, z)=\operatorname{grad} f(x) + \sum_{i =1 }^{l} y_{i} \operatorname{grad} h_{i}(x) + \sum_{i =1 }^{m} z_{i} \operatorname{grad} g_{i}(x),
\end{equation*}
where $\{\operatorname{grad} h_{i}(x)\}_{i=1}^{l}$ and $ \{\operatorname{grad} g_{i}(x)\}_{i=1}^{m}$ are the gradients of the component functions of $h$ and $g$. 
The Riemannian KKT conditions (e.g., see \cite[Definition 2.3]{liu2020simple} or \cite{yang2014optimality}) for (\ref{RCOP}) are given by
\begin{equation}\label{eq:KKTConditions}
	\operatorname{grad}_{x} \mathcal{L}(x, y, z) =0_{x}; \;
	h(x) =0 ,
	g(x) \leq 0 , 
	z  \geq 0; \;
	Z g(x) =0.
\end{equation}
The above conditions can be written in terms of slack variables $\mathbb{R}^{m} \ni s:=-g(x)$, as
\begin{equation}\label{eq:KKTVectorField}
	F(w):=\left(\begin{array}{l}
		\operatorname{grad}_{x} \mathcal{L}(x, y, z) \\
		 h(x) \\
		g(x)+s \\
		Z S e
	\end{array}\right)=0_{w}
	=\left(\begin{array}{l}
		0_{x} \\
		0\\
		0\\
		0
	\end{array}\right) \in T_w \mathcal{N},
\end{equation}
and $(z,s) \geq 0,$
where $w:=(x, y, z,s) \in \mathcal{N}:=\mathcal{M}\times \mathbb{R}^{l} \times \mathbb{R}^{m} \times \mathbb{R}^{m}$.
Here, we have generated a vector field $F$ in (\ref{eq:KKTVectorField}) on the product manifold $\mathcal{N}$, i.e.,
for any $w\in \mathcal{N}$, 
$
F(w) \in
T_{w} \mathcal{N} \cong T_{x} \mathcal{M} \times \mathbb{R}^{l} \times \mathbb{R}^{m} \times \mathbb{R}^{m}.
$
We will call $F \colon \mathcal{N} \to T\mathcal{N} $ in (\ref{eq:KKTVectorField}) above the \emph{KKT vector field} for (\ref{RCOP}).
Note that for $\xi=\left(\xi_x, \xi_y, \xi_z, \xi_s\right) \in T_{w} \mathcal{N} $,
\begin{equation}\label{eq:barR}
	\bar{\mathrm{R}}_{w}(\xi):=(\mathrm{R}_{x}(\xi_{x}), y+\xi_{y}, z+\xi_{z}, s+\xi_{s})
\end{equation}
is a well-defined retraction on the product manifold $\mathcal{N}$ as long as $\mathrm{R}$ is a retraction on $\mathcal{M}$.

Now, we aim to find a singularity of the KKT vector field $F$ on $\mathcal{N}$ under some nonnegative constraints $(z,s) \geq 0$.
If the Riemannian Newton method is to be applied to (\ref{eq:KKTVectorField}), we must formulate the covariant derivative of $F$ at an arbitrary $w \in \mathcal{N}$. Before that, we need some new symbols.
Fixing a point $x \in \mathcal{M}$, we can define two linear operators $\mathcal{H}_{x}\colon \mathbb{R}^{l} \rightarrow T_{x} \mathcal{M}$ and $\mathcal{G}_{x}\colon \mathbb{R}^{m} \rightarrow T_{x} \mathcal{M}$ by
\begin{equation}\label{defn:HxGx}
	\mathcal{H}_x [v]:=\sum_{i=1}^l v_i \operatorname{grad} h_i(x), \quad \mathcal{G}_x [u]:=\sum_{i=1}^m u_i \operatorname{grad} g_i(x),
\end{equation}
respectively.
Then, for $\mathcal{H}_{x}$, its adjoint operator $\mathcal{H}_x^*\colon T_x \mathcal{M} \rightarrow \mathbb{R}^l$ is given by
\begin{equation*}
	\mathcal{H}_x^* [\xi]=\left(\left\langle\operatorname{grad} h_1(x), \xi\right\rangle_x, \cdots,\left\langle\operatorname{grad} h_l(x), \xi\right\rangle_x\right)^T.
\end{equation*}
Also, $\mathcal{G}_x^*\colon T_x \mathcal{M} \rightarrow \mathbb{R}^m$ with $\mathcal{G}_x^* [\xi]=\left(\left\langle\operatorname{grad} g_1(x), \xi\right\rangle_x, \cdots,\left\langle\operatorname{grad} g_m(x), \xi\right\rangle_x\right)^T$. 

Now, using the solutions of Exercises 5.4 and 5.13 on covariant derivatives of vector fields on product manifolds in monograph \cite{boumal2022intromanifolds}, the following results can be easily derived.
Given KKT vector field $F $ in (\ref{eq:KKTVectorField}) and $w \in \mathcal{N}$, the covariant derivative of $F$ at $w$ is the operator $\nabla F(w)\colon T_{w} \mathcal{N} \to T_{w} \mathcal{N}$ given by
\begin{equation}\label{eq:CovariantDerivativeKKTVectorField}
	\nabla F(w) [\Delta w] =
	\left(\begin{array}{l}
		\operatorname{Hess}_{x} \mathcal{L}(w) [\Delta x]
		+\mathcal{H}_{x} [\Delta y]
		+\mathcal{G}_{x} [\Delta z]\\
		\mathcal{H}_{x}^{*}[\Delta x]\\
		\mathcal{G}_{x}^{*}[\Delta x]+\Delta s\\
		Z \Delta s + S \Delta z
	\end{array}\right),
\end{equation}
where $\Delta w=(\Delta x, \Delta y, \Delta z, \Delta s) \in T_{w} \mathcal{N}$ and $\operatorname{Hess}_{x} \mathcal{L}(w)$ is the Riemannian Hessian of real-valued function $\mathcal{L}(\cdot, y, z)$.
Indeed,
$\operatorname{Hess}_{x} \mathcal{L}(w)\colon T_{x}\mathcal{M} \to T_{x}\mathcal{M}$ satisfies
\begin{equation}\label{eq:HessianLagrange}
	\operatorname{Hess}_{x} \mathcal{L}(w) =
	\operatorname{Hess} f(x)
	+\sum_{i=1}^{l} y_{i} \operatorname{Hess} h_{i}(x)
	+\sum_{i=1}^{m} z_{i} \operatorname{Hess} g_{i}(x),
\end{equation}
where $\{\operatorname{Hess} h_{i}(x)\}_{i=1}^{l}$ and $\{\operatorname{Hess} g_{i}(x)\}_{i=1}^{m}$ are Hessians of the component functions of $h$ and $g$.

\subsection{Implication of Standard Assumptions}

Next, we will discuss an important result about the covariant derivative of the KKT vector field given in (\ref{eq:CovariantDerivativeKKTVectorField}).
Let $\mathbb{A}(x):=\left\{i \mid g_{i}(x)=0\right\}$ denote the active set at $x \in \mathcal{M}$.
A prior study on Riemannian optimality conditions \cite{yang2014optimality,bergmann2019intrinsic} showed that the following assumptions are meaningful.
We call them the \emph{standard Riemannian assumptions} for (\ref{RCOP}).
Note that the $x^{*}$ and $w^*$ in \ref{B3}-\ref{B5} all refer to those in \ref{B1}.

\setlist[assumptions,1]{label=\textbf{(A\arabic*)}}
\begin{assumptions}[leftmargin=9mm]
	\item {Existence.} There exists $w^*=(x^{*}, y^{*}, z^{*}, s^{*})$ satisfying the Riemannian KKT conditions (\ref{eq:KKTConditions}). Here, we introduce the slack variables $s^{*}:=-g(x^{*})$.
	\label{B1}	
	\item {Linear Independence Constraint Qualification (LICQ).} \\ The set $\left\{\operatorname{grad} h_{i}(x^{*})\right\}_{i=1}^{l} \bigcup \left\{\operatorname{grad} g_{i}(x^{*})\right\}_{i \in \mathbb{A}(x^{*})}$ is linearly independent in $ T_{x^{*}}\mathcal{M}$.
	\label{B3}	
	\item {Strict complementarity.} $(z^{*})_{i}>0$ if $g_{i}(x^{*})=0$ for all $i=1, \cdots,m$.
	\label{B4}	
	\item {Second-order sufficiency.}
	$\left \langle\operatorname{Hess}_{x} \mathcal{L}(w^{*})[\xi], \xi\right\rangle>0$ for all nonzero $\xi \in T_{x^{*}}\mathcal{M}$ \\ satisfying $\left\langle\xi, \operatorname{grad} h_{i}(x^{*})\right\rangle=0 $ for $i=1,\dots,l,$ and
	$\left\langle\xi, \operatorname{grad} g_{i}(x^{*})\right\rangle=0$ for $ i \in \mathbb{A}(x^{*})$. 
	\label{B5}
\end{assumptions}

Recall the Riemannian Newton method discussed in Section \ref{subsec:RiemannianNewton}.
It can be shown that if $p^*$ is a solution of equation (\ref{eq:Singularity}) and the covariant derivative $\nabla F\left(p^*\right)$ is \emph{nonsingular}, then Algorithm \ref{algo:StandardNewton} has the local superlinear convergence \cite{fernandes2017superlinear} and local quadratic convergence \cite{ferreira2002kantorovich} under certain mild conditions on the map $p \mapsto \nabla F(p)$.
Note that the requirement of nonsingularity for the covariant derivative at the solution point is of primary importance.
Therefore, the next theorem motivates us to use the Riemannian Newton method to solve (\ref{eq:KKTVectorField}).

\begin{theorem}\label{prop:non}
	If the standard Riemannian assumptions \ref{B1}-\ref{B5} hold at some point $w^{*}$,
	then the operator $\nabla F(w^{*})$ in (\ref{eq:CovariantDerivativeKKTVectorField}) is nonsingular.
\end{theorem}
\begin{proof}
This proof omits all the asterisks of the variables. Define $ \mathbb{E}:=\{1,\dots,l\}$ and $\mathbb{I}:=\{1,\dots,m\} $. Take some $w=(x, y, z, s) \in \mathcal{N} $ satisfying \ref{B1}-\ref{B5}, then we have $ s_{i} = -g_{i}(x)$ and $z_{i} s_{i} = 0$ for all $ i \in \mathbb{I}$. For short, let $\mathbb{A}:=\mathbb{A}(x) \subset \mathbb{I} $. 
Suppose that $\nabla {F}(w) [\Delta w] =0$ for some $\Delta w = (\Delta x, \Delta y, \Delta z, \Delta s) \in T_{w} \mathcal{N} \cong T_x \mathcal{M} \times \mathbb{R}^l \times \mathbb{R}^m \times \mathbb{R}^m$.
$\Delta y_i$ denotes the components of the vector $\Delta y$, as do $\Delta z_i$, $\Delta s_i$.
To prove its nonsingularity, we will show that $ \Delta w=0$.
Expanding the equation $\nabla {F}(w) [\Delta w] =0$ gives
\begin{equation}\label{13}
	\left\{
	\begin{aligned}
		0&= \operatorname{Hess}_{x} \mathcal{L}(w) [\Delta x]
		+\sum_{i\in \mathbb{E}} \Delta y_{i} \operatorname{grad} h_{i}(x)
		+\sum_{i\in \mathbb{I}} \Delta z_{i} \operatorname{grad} g_{i}(x),\\
		0&=\left\langle\operatorname{grad}h_{i}(x), \Delta x\right\rangle, \text { for all } i\in \mathbb{E},\\
		0&=\left\langle\operatorname{grad} g_{i}(x), \Delta x\right\rangle + \Delta s_{i}, \text { for all } i\in \mathbb{I}, \\
		0&=z_{i} \Delta s_{i} + s_{i} \Delta z_{i}, \text { for all } i \in \mathbb{I}. 
	\end{aligned}
	\right.
\end{equation}
Strict complementarity \ref{B4} and the last equalities above imply that $ \Delta s_{i}=0 $
for all $ i \in \mathbb{A}$ and $ \Delta z_{i}=0 $
for $ i \in  \mathbb{I} \setminus \mathbb{A}$.
Substituting those values into the system (\ref{13}) reduces it to
\begin{equation}\label{sing1}
	\left\{
	\begin{aligned}
		0&= \operatorname{Hess}_{x} \mathcal{L}(w) [\Delta x]
		+\sum_{i\in \mathbb{E}} \Delta y_{i} \operatorname{grad} h_{i}(x)
		+\sum_{i\in \mathbb{A}} \Delta z_{i} \operatorname{grad} g_{i}(x), \\
		0&=\left\langle\operatorname{grad}h_{i}(x), \Delta x\right\rangle, \text { for all } i\in \mathbb{E},\\
		0&=\left\langle\operatorname{grad}g_{i}(x), \Delta x\right\rangle, \text { for all } i\in \mathbb{A},
	\end{aligned}
	\right.
\end{equation}
and
$
\Delta s_{i}=-\left\langle\operatorname{grad} g_{i}(x), \Delta x\right\rangle \text { for all } i\in \mathbb{I} \setminus \mathbb{A}.
$
It follows from system (\ref{sing1}) that
\begin{align*}
	0
	&=  \langle \operatorname{Hess}_{x} \mathcal{L}(w) [\Delta x]
	+\sum_{i\in \mathbb{E}} \Delta y_{i} \operatorname{grad} h_{i}(x)
	+\sum_{i\in  \mathbb{A}} \Delta z_{i} \operatorname{grad} g_{i}(x) , \Delta x \rangle \\
	&= \left\langle\operatorname{Hess}_{x} \mathcal{L}(w) [\Delta x], \Delta x\right\rangle
	+\sum_{i\in \mathbb{E}} \Delta y_{i} \left\langle\operatorname{grad}h_{i}(x), \Delta x\right\rangle
	+\sum_{i\in  \mathbb{A}} \Delta z_{i} \left\langle\operatorname{grad}g_{i}(x), \Delta x\right\rangle \\
	&= \left\langle\operatorname{Hess}_{x} \mathcal{L}(w) [\Delta x], \Delta x\right\rangle.
\end{align*}
Thus, from second-order sufficiency \ref{B5}, $ \Delta x$ must be zero elements.
And then $\Delta s_{i}=0 \text { for all } i\in \mathbb{I} \setminus  \mathbb{A}.$
Next, substituting $ \Delta x=0$ into the first equation in (\ref{sing1}) yields
$
0=\sum_{i\in \mathbb{E}} \Delta y_{i} \operatorname{grad} h_{i}(x)
+\sum_{i\in  \mathbb{A}} \Delta z_{i} \operatorname{grad} g_{i}(x).
$
The LICQ \ref{B3} implies that the coefficients $\Delta y_{i}$ for $i\in \mathbb{E}$ and $ \Delta z_{i} $ for $ i\in  \mathbb{A} $ must be zero. 
This completes the proof.
\end{proof}

\subsection{Prototype Algorithm of RIPM}

Applying the Riemannian Newton method directly to the KKT vector field $F: \mathcal{N} \rightarrow T \mathcal{N}$ results in the following Newton equation (see (\ref{eq:StandardNewtonEquation}) without iteration count $k$) at each iteration:
\begin{equation}\label{eq:OriginalNewtonEquation}
	\nabla F(w) [\Delta w]+F(w)=0.
\end{equation}
As with the usual interior point method in the Euclidean setting, once the iterates reach the boundary of the feasible region, they are forced to stick to it \cite[P6]{wright1997primal}.
For the iterates to maintain a sufficient distance from the boundary, we introduce a perturbed complementary equation with some barrier parameter $\mu>0$ and define the \textit{perturbed KKT vector field}:
\begin{equation}\label{eq:perturbedKKTvectorfield}
	F_{\mu}(w):=F(w)-\mu \hat{e}, \text { and } \hat{e}:=\hat{e}(w):=\left(0_x, 0,0, e\right).
\end{equation}
Notice that the perturbation term $\hat{e}$, indeed, is a special vector field on $\mathcal{N}$, not a constant, because $0_x$ is essentially dependent on $w$ and/or $x$.
In fact, the covariant derivative of the perturbed KKT vector field is the same as that of the original. From the linearity of the connection $\nabla$, for any $w \in \mathcal{N}$ and any $\mu>0$, we have 
\begin{equation}\label{eq:NabulaperturbedKKTvectorfield}
	\nabla F_{\mu}(w)=\nabla F(w)-\mu \nabla \hat{e}(w)=\nabla F(w),
\end{equation}
where the last equity comes from $\nabla \hat{e}(w) [\Delta w]=\left(0_x, 0,0,0\right)$ for all $\Delta w \in T_w \mathcal{N}$.
Applying the Riemannian Newton method to perturbed KKT vector field $F_{\mu}(w)$ yields the \emph{perturbed Newton equation}: $\nabla F_{\mu}(w) [\Delta w]+F_{\mu}(w)=0$. 
From (\ref{eq:perturbedKKTvectorfield}) and (\ref{eq:NabulaperturbedKKTvectorfield}), this equation is equivalent to 
\begin{equation}
	\nabla F(w) [\Delta w]+F(w)=\mu \hat{e},
\end{equation}
which reduces to the ordinary Newton equation (\ref{eq:OriginalNewtonEquation}) as $\mu \to 0$.
At this point, we can describe a prototype of the Riemannian Interior Point Method (RIPM) in Algorithm \ref{algo:prototype_algo}.

\begin{algorithm}
	\caption{Prototype Algorithm of RIPM for (\ref{RCOP})}
	\label{algo:prototype_algo}
	\SetAlgoLined
	
	\KwIn{An initial point $w_0=\left(x_0, y_0, z_0, s_0\right) \in \mathcal{N}$ with $\left(z_0, s_0\right)>0$ and a retraction $\mathrm{R}$ on $\mathcal{M}$. $\hat{\gamma} \in(0,1)$, $\mu_0>0$.}
	\KwOut{Sequence $\left\{w_k\right\} \subset \mathcal{N}$ such that $\left\{w_k\right\} \rightarrow w^*$ and $w^*$ satisfies the KKT conditions (\ref{eq:KKTConditions}).} %
	
	Set $k \to 0$\;
	
	\While{Stopping criterion not satisfied}{
		1. Obtain $\Delta w_{k}=(\Delta x_{k}, \Delta y_{k}, \Delta z_{k},\Delta s_{k}) \in T_{w_k} \mathcal{N}$ by solving the perturbed Newton equation on $T_{w_k} \mathcal{N}$:
		\begin{equation}\label{eq:prototype}
			\nabla F(w_{k}) [\Delta w_{k}]=-F(w_{k})+\mu_{k} \hat{e};
		\end{equation}
	
		2. Choose $ \gamma_{k} \in [\hat{\gamma},1]$ and compute the step size $\alpha_{k}$ defined by
		\begin{equation}\label{rule:StepSizeA}
			\min \left\{1, \gamma_{k} \min _{i}\left\{-\frac{(s_{k})_{i}}{(\Delta s_{k})_{i}} \mid(\Delta s_{k})_{i}<0\right\},\gamma_{k} \min _{i}\left\{-\frac{(z_{k})_{i}}{(\Delta z_{k})_{i}} \mid(\Delta z_{k})_{i}<0\right\}\right\};
		\end{equation}
	
		3. Compute the next point as
		$w_{k+1}:=\bar{\mathrm{R}}_{w_k}\left(\alpha_k \Delta w_k\right)$, see (\ref{eq:barR})\;
		
		4. Choose $0<\mu_{k+1}<\mu_k$\;
		
		5. $k \to k+1$\;
	}
\end{algorithm}

In the step 2 of Algorithm \ref{algo:prototype_algo}, we just wish to compute a step size $0<\alpha_k \leq 1$ to ensure new point $w_{k+1}$ with $\left(z_{k+1},s_{k+1}\right)>0$. 
There are many schemes to achieve this purpose. 
The scheme in (\ref{rule:StepSizeA}) is simple but sufficient to guarantee the local convergence, as will be proved in Section \ref{sec:local}.

\subsection{Solving Perturbed Newton Equation Efficiently}

The challenge of Algorithm \ref{algo:prototype_algo} is how to solve the Newton equation (\ref{eq:prototype}) in an efficient manner.
In this subsection, we will do this in two steps: the first step will be to turn the original full Newton equation, which is asymmetric and consists of four variables, into a condensed form, which is symmetric and consists of only two variables. In the second step, an iterative method, namely, the Krylov subspace method, is used to solve the \emph{operator equations} directly, avoiding the expensive computational effort of converting them into the usual \emph{matrix equations}.

\subsubsection{Condensed Form of Perturbed Newton Equation}

Let us consider Algorithm \ref{algo:prototype_algo} and omit the iteration count $k$.
Given the current point $w \in \mathcal{N}$ with $(z, s) > 0$, for the KKT vector field $F(w)$ in (\ref{eq:KKTVectorField}), we denote its components by $F_x,F_y,F_z,F_s$ in top-to-bottom order, namely,
$F_x:=\operatorname{grad}_x \mathcal{L}(w),F_y:=h(x),F_z:=g(x)+s,F_s:=Z S e.$
By using these symbols and the formulation of $\nabla F(w)$ in (\ref{eq:CovariantDerivativeKKTVectorField}), the \emph{full} (perturbed) Newton equation (\ref{eq:prototype}) defined on $T_w \mathcal{N} \cong T_x \mathcal{M} \times \mathbb{R}^l \times \mathbb{R}^m \times \mathbb{R}^m$ is expanded as:
\begin{equation}\label{eq:FullNewtonEquation}
	\left(\begin{array}{l}
		\operatorname{Hess}_{x} \mathcal{L}(w) [\Delta x]
		+\mathcal{H}_{x} [\Delta y]
		+\mathcal{G}_{x} [\Delta z]\\
		\mathcal{H}_{x}^{*}[\Delta x]\\
		\mathcal{G}_{x}^{*}[\Delta x]+\Delta s\\
		Z \Delta s + S \Delta z
	\end{array}\right)
	=
	\left(\begin{array}{l}
		-F_x\\
		-F_y\\
		-F_z\\
		-F_s+\mu e
	\end{array}\right).
\end{equation}
Not only does this equation contain four variables, but there is no symmetry on the left side of the equation, so it would be unwise to solve it just like that.
By $(z, s) > 0$ and the fourth line of (\ref{eq:FullNewtonEquation}), we deduce
$
\Delta s=Z^{-1}\left(\mu e -F_s -S \Delta z \right).
$
Substituting this $\Delta s$ into the third line of (\ref{eq:FullNewtonEquation}), we get
$ \mathcal{G}_{x}^{*}[\Delta x]-Z^{-1} S \Delta z = -Z^{-1} \mu e-g(x)$ and thus $ \Delta z=S^{-1}\left[Z\left(\mathcal{G}_{x}^{*}[\Delta x]+F_z\right)+\mu e -F_s\right] $.
Substituting this $\Delta z$ further into the first line of (\ref{eq:FullNewtonEquation}) and combining it with the second line of (\ref{eq:FullNewtonEquation}) yields the following \emph{condensed} Newton equation, which is defined on $T_{x} \mathcal{M} \times \mathbb{R}^{l}$:
\begin{equation}\label{eq:CondensedEquation}
	\mathcal{T}(\Delta x,\Delta y):=
	\left(\begin{array}{l}
		\mathcal{A}_{w} [\Delta x]
		+\mathcal{H}_{x} [\Delta y]\\
		\mathcal{H}_{x}^{*} [\Delta x]
	\end{array}\right)
	=
	\left(\begin{array}{l}
		c\\
		q
	\end{array}\right).
\end{equation}
where
\begin{equation}\label{eq:SemiEquationWhere}
	\begin{aligned}
		\mathcal{A}_{w}&:=\operatorname{Hess}_{x} \mathcal{L}(w)+
		\mathcal{G}_{x} S^{-1} Z \mathcal{G}_{x}^{*}, \\
		c&:=
		-F_x- \mathcal{G}_{x} S^{-1}\left( Z F_z+\mu e-F_s\right),
		q:=-F_y.
	\end{aligned}
\end{equation}
Here, $c$ and $q$ are constant vectors. If we defined $\Theta:=\mathcal{G}_x S^{-1} Z \mathcal{G}_x^*$, then $\mathcal{A}_{w}=\operatorname{Hess}_{x} \mathcal{L}(w)+ \Theta.$
Note that both $\Theta$ and $\operatorname{Hess}_{x} \mathcal{L}(w)$ are operators from and to $T_x \mathcal{M}.$

From the discussion above, for any $w \in \mathcal{N}$ with $(z, s)>0$, the operator $\nabla F(w)$ in (\ref{eq:CovariantDerivativeKKTVectorField}) is nonsingular if and only if the newly defined operator $\mathcal{T}$ in (\ref{eq:CondensedEquation}) is nonsingular.
Eventually, it is sufficient for us to solve the equation (\ref{eq:CondensedEquation}) containing only two variables $\Delta x$ and $\Delta y$.
In fact, when we consider the case of only inequality constraints in (\ref{RCOP}), then $\Delta y$ vanishes, and only a linear equation 
\begin{equation}\label{eq:OnlyInequality}
	\mathcal{A}_w [\Delta x]=c
\end{equation}
on $T_{x} \mathcal{M}$ needs to be solved.
More importantly, the operator $\mathcal{T}$ in the left side of (\ref{eq:CondensedEquation}) is symmetric, or say self-adjoint (although often indefinite).
It is trivial to check that operators $\Theta $ and $\mathcal{A}_{w}$ are self-adjoint on $T_{x} \mathcal{M}$; and thus $\mathcal{T}$ is self-adjoint on the product vector space $T_{x} \mathcal{M} \times \mathbb{R}^{l}$ equipped with inner product
$\left\langle(\xi_x, \xi_y),(\eta_x, \eta_y)\right\rangle:=\left\langle \xi_x, \eta_x\right\rangle_x+ \xi_y ^T \eta_y$.
We can also see that (\ref{eq:CondensedEquation}) is a saddle point problem defined on Hilbert spaces form its special structure. 

\subsubsection{Krylov Subspace Methods on Tangent Space}
\label{ssec:Krylov}
Next, how to solve (\ref{eq:CondensedEquation}) efficiently becomes critical.
For simplicity, we consider the case of only inequality constraints in (\ref{RCOP}), then we will solve operator equation (\ref{eq:OnlyInequality}) with a self-adjoint operator $\mathcal{A}_w \colon T_x \mathcal{M} \to T_x \mathcal{M}$.
Let $d:=\operatorname{dim} T_x \mathcal{M}$.
Unfortunately, in most cases of practical applications, the Riemannian situation leaves us with no explicit matrix form available for $\mathcal{A}_w$.
This means that we can only access $\mathcal{A}$ (subscript $w$ omitted) by inputting a vector $v$ to return $\mathcal{A}v$.

A general approach is first to find the matrix representation $\hat{\mathcal{A}}$ for $\mathcal{A}$ under some basis of $T_x \mathcal{M}$. 
In detail, the full process of this approach is as follows:
\begin{enumerate}[label=\textbf{(Step \arabic*)}, leftmargin=14mm]
	\item Obtain $d$ random independent vectors on $T_x \mathcal{M}$.
	\item Obtain an orthonormal basis $\left\{u_i\right\}_{i=1}^{d}$ of $T_x \mathcal{M}$ by the modified Gram-Schmidt algorithm.
	\item Compute $[\hat{\mathcal{A}}]_{i j}:=\left\langle\mathcal{A} u_j, u_i\right\rangle_x$ for $1 \leq i \leq j \leq d$ due to symmetry, then we obtain the matrix representation $\hat{\mathcal{A}} \in \mathbb{R}^{d \times d}$.
	\item Compute $[\hat{c}]_{i}:=\left\langle c, u_i \right\rangle_x$ for $1 \leq i \leq d$, then we obtain the vector representation $\hat{c} \in \mathbb{R}^{d}$.
	\item Using arbitrary linear solver to get solution $\Delta \hat{x} \in \mathbb{R}^d$ from matrix equation
	$\hat{\mathcal{A}} \Delta \hat{x}=\hat{c}$.
	\item Recovery the tangent vector $\Delta x \in T_x \mathcal{M}$ by $\Delta x = \sum_{i=1}^{d} (\Delta \hat{x})_i u_i$.
\end{enumerate}
In Algorithm \ref{algo:prototype_algo}, at each iteration, $x$ is updated, and thus, the tangent space $T_x \mathcal{M}$ changes. Thus, the above six steps need to be done all over again. Obviously, this approach is so expensive that it is not feasible in practice.

An ideal approach is to use an iterative method, such as the Krylov subspace method (e.g., conjugate gradients method \cite[Chapter 6.3]{boumal2022intromanifolds}), on $T_x \mathcal{M}$ directly. Such a method does not explicitly require a matrix representation $\hat{\mathcal{A}}$ for $\mathcal{A}$. In general, it only needs to call an abstract linear operator $v \mapsto \mathcal{A} v$. 
Since $\mathcal{A}$ in (\ref{eq:SemiEquationWhere}) is self-adjoint but indefinite, for solving operator equation (\ref{eq:OnlyInequality}), we will use the Conjugate Residual (CR) method (see \cite[ALGORITHM 6.20]{saad2003}) as stated in Algorithm \ref{algo:ConjugateResidualTangentSpace}.

\begin{algorithm}
	\caption{Conjugate Residual (CR) Method on Tangent Spaces for (\ref{eq:OnlyInequality})}
	\label{algo:ConjugateResidualTangentSpace}
	\SetAlgoLined
	
	\KwIn{Symmetric invertible linear operator $\mathcal{A}: T_x \mathcal{M} \to T_x \mathcal{M}$, nonzero $c \in T_x \mathcal{M}$ and an initial point $v_0 \in T_x \mathcal{M}$.}
	\KwOut{Sequence $\left\{v_n\right\} \subset T_x \mathcal{M}$ such that $\left\{v_n\right\} \to v^{*}$ and  $\mathcal{A} v^{*}=c$.}
	
	Set $n \to 0$, $r_0:=c-A v_0, p_0:=r_0$ and compute $\mathcal{A}r_0, \mathcal{A}p_0$ \;
	
		\While{stopping criterion not satisfied}{
			1. Update number $\alpha_n := \langle r_n, \mathcal{A} r_n \rangle _{x} / \langle \mathcal{A}p_n, \mathcal{A}p_n \rangle _{x} $ \tcp*{Step length}
			2. Set $v_{n+1} := v_n + \alpha_n p_n$ \tcp*{Iterate point}
			3. Update $r_{n+1} := r_n - \alpha_n \mathcal{A} p_n$ \tcp*{Residual} 
			4. Compute $\mathcal{A} r_{n+1}$ \tcp*{This is the only call to $\mathcal{A}$ in while loop}
			5. Update number $\beta_n := \langle r_{n+1}, \mathcal{A} r_{n+1} \rangle _{x} / \langle r_n, \mathcal{A} r_n \rangle _{x} $ \;
			6. Set $p_{n+1} := r_{n+1} + \beta_n p_n$ \tcp*{Conjugate direction}
			7. Compute $\mathcal{A} p_{n+1} :=\mathcal{A} r_{n+1}+\beta_n \mathcal{A} p_n$ \tcp*{No need to call $\mathcal{A}$ here}
			8. $n \to n + 1$\;
		}
\end{algorithm}
A significant feature is that the iterate points $v_k$, conjugate directions $p_k$, and residual vectors $r_k:=\mathcal{A} v_k-c$ are all contained in $T_x \mathcal{M}$. 
Usually, the initial point $v_0$ is the zero element of $T_x \mathcal{M}$; the iteration terminates when the relative residual $\left\|r_k\right\| /\|c\| \leq \epsilon$ for some threshold $\epsilon>0$, or some maximum number of iterates is reached.

The discussion of the above two approaches can be naturally extended to the case containing equality constraints in (\ref{RCOP}), where we consider $\mathcal{T}$ on $T_x \mathcal{M} \times \mathbb{R}^l$ instead of $\mathcal{A}$ on $T_x \mathcal{M}$.

\section{Preliminaries and Auxiliary Results}\label{sec:Preliminaries}

This section introduces the useful results that are indispensable to our subsequent discussion.

\begin{remark}
	For a retraction $\mathrm{R}$ on $\mathcal{M}$ and $x \in \mathcal{M}$,
	by $\mathcal{D} \mathrm{R}_{x}\left(0_{x}\right)=\mathrm{id}_{T_{x} \mathcal{M}}$ and the inverse function theorem, there exists a neighborhood $V$ of $0_{x}$ in $T_{x} \mathcal{M}$ such that $ \mathrm{R}_{x} $ is a diffeomorphism on $V$; thus,
	$\mathrm{R}_{x}^{-1}(y)$ is well defined for all $ y \in \mathcal{M} $ sufficiently close to $ x $.
	In this case, $\mathrm{R}_{x}(V) \subset \mathcal{M}$ is called a retractive neighborhood of $x$.
	Furthermore, the existence of a \emph{totally} retractive neighborhood \cite[Theorem 2]{zhu2020riemannian} shows that for any $\bar{x} \in \mathcal{M} $ there is a neighborhood $W$ of $\bar{x}$ such that
	$\mathrm{R}_{x}^{-1}(y)$ is well defined for all $x, y \in W$.
	In what follows, we will suppose that an appropriate neighborhood has been chosen by default for the well-definedness of $\mathrm{R}_x^{-1}(y)$.
\end{remark}

\subsection{Vector Transport and Parallel Transport}

Define the Whitney sum
$
T \mathcal{M} \oplus T \mathcal{M}:=\left\{(\eta, \xi): \eta, \xi \in T_x \mathcal{M}, x \in \mathcal{M} \right\}.
$
A smooth map
$
\mathrm{T} \colon T \mathcal{M} \oplus T \mathcal{M} \to T \mathcal{M}: (\eta, \xi) \mapsto \mathrm{T}_{\eta}(\xi),
$
is called a \emph{vector transport} on $\mathcal{M}$ if there exists an associated retraction $\mathrm{R}$ on $\mathcal{M}$ such that $\mathrm{T}$ satisfies the following properties for all $x \in \mathcal{M}$:
\begin{enumerate}
	\item Associated retraction. $\mathrm{T}_{\eta}\left(\xi\right) \in T_{\mathrm{R}_{x}\left(\eta\right)} \mathcal{M}$
	for all $\eta, \xi \in T_x \mathcal{M}$.	
	\item Consistency. $\mathrm{T}_{0_{x}}\left(\xi\right)=\xi$ for all $\xi \in T_{x} \mathcal{M}$.	
	\item Linearity.
	$\mathrm{T}_{\eta}\left(a \xi+b \zeta\right)=a \mathrm{T}_{\eta}\left(\xi\right)+b \mathrm{T}_{\eta}\left(\zeta\right)$ for all $a, b \in \mathbb{R}$ and $\eta, \xi, \zeta \in T_{x} \mathcal{M}$.
\end{enumerate}
Thus, fixing any $\eta \in T_x \mathcal{M}$, the map
$
\mathrm{T}_\eta \colon T_x \mathcal{M} \to T_{\mathrm{R}_x(\eta)} \mathcal{M}: \xi \mapsto \mathrm{T}_\eta(\xi),
$
is a linear operator.
Additionally, $\mathrm{T}$ is \emph{isometric} if 
$
\langle \mathrm{T}_{\eta} (\xi), \mathrm{T}_{\eta} (\zeta)\rangle =\langle\xi, \zeta\rangle
$
holds, for all $x \in \mathcal{M}$ and all $\xi, \zeta, \eta \in T_{x} \mathcal{M}$.
In other words, for any $\eta \in T_x \mathcal{M}$, the adjoint and the inverse of $\mathrm{T}_{\eta}$ coincide, i.e., $ \mathrm{T}_{\eta}^{*} =  \mathrm{T}_{\eta}^{-1}$.
There are two important classes of vector transport as follows. Let $\mathrm{R}$ be a retraction on $\mathcal{M}$. 
\begin{enumerate}[(1)]
	\item 
	The \emph{differentiated retraction} defined by
	\begin{equation}\label{eq:DifferentiatedRetraction}
		\mathrm{T}_{\eta}(\xi):=\mathcal{D} \mathrm{R}_{x}(\eta)[\xi],  \; \eta, \xi \in T_{x} \mathcal{M}, x \in \mathcal{M}, 
	\end{equation}
	is a valid vector transport \cite[equation (8.6)]{absil2009optimization}.	
	\item
	Given a smooth curve $\gamma \colon [0,1] \to \mathcal{M}$ and $t_0, t_1 \in [0,1]$, the \emph{parallel transport} from the tangent space at $\gamma(t_0)$ to the tangent space at $\gamma(t_1)$ along $\gamma,$
	is a linear operator
	$ \mathrm{P}_\gamma^{t_1 \to t_0} \colon T_{\gamma\left(t_0\right)} \mathcal{M} \rightarrow T_{\gamma\left(t_1\right)} \mathcal{M}$ defined by $\mathrm{P}_\gamma^{t_1 \to t_0}(\xi)=Z\left(t_1\right)$, where $Z$ is the unique parallel vector field such that $Z\left(t_0\right)=\xi$.
	Then, for any $x \in \mathcal{M}, \eta \in T_x \mathcal{M}$, then
	\begin{equation}\label{eq:ParallelTransport}
		\mathrm{T}_{\eta}\left(\xi\right):=\mathrm{P}_\gamma^{1 \to 0} (\xi)
	\end{equation}
	is a valid vector transport \cite[equation (8.2)]{absil2009optimization}, where $\mathrm{P}_\gamma$ denotes the parallel transport along the curve $t \mapsto \gamma(t):=\mathrm{R}_x\left(t \eta\right)$. We often omit the superscript $^{1 \to 0}$ if it is clear from the context.
	In particular, parallel transport is isometric.
\end{enumerate}

\subsection{Lipschitz Continuity with respect to Vector Transports}\label{subsec:LipschitzContinuity}

Multiple Riemannian versions of Lipschitz continuity have been defined, e.g., \cite[Section 10.4]{boumal2022intromanifolds}.
Here, we consider Lipschitz continuity with respect to a vector transport.
In what follows, let $\mathcal{M}$ be a Riemannian manifold endowed with a vector transport $\mathrm{T}$ and an associated retraction $\mathrm{R}$.
We first consider the Lipschitz continuous gradient of scale field $f$.
\begin{definition}[{\cite[Definition 5.2.1]{huang2013optimization}}]
	\label{defn:LipschitzContinuouslyDifferentiable}
	A function $f\colon \mathcal{M} \rightarrow \mathbb{R}$ is \emph{Lipschitz continuously differentiable with respect to $\mathrm{T}$ in $\mathcal{U} \subset \mathcal{M}$} if it is differentiable and there exists a constant $\kappa>0$ such that, for all $x, y \in \mathcal{U}$,
	$
	\left\|\operatorname{grad} f(y)-\mathrm{T}_{\eta} [\operatorname{grad} f(x)] \right\| \leq \kappa\|\eta\|,
	$
	where $\eta=\mathrm{R}_{x}^{-1} y$. 
\end{definition}
%

Going one degree higher, let us now discuss the Lipschitz continuity of Hessian operators.
Throughout this paper, for a linear operator $\mathcal{A}: \mathcal{E} \rightarrow \mathcal{E}^{\prime}$ between two finite-dimensional normed vector spaces $\mathcal{E}$ and $\mathcal{E}^{\prime}$, the (operator) norm of $\mathcal{A}$ is defined by $\|\mathcal{A}\|:=\sup \left\{\|\mathcal{A} v\|_{\mathcal{E}^{\prime}} : v \in \mathcal{E},\|v\|_{\mathcal{E}}=1\right.$, or, $\left.\|v\|_{\mathcal{E}} \leq 1\right\}.$
The inverse of $\mathrm{T}_{\eta}$ is needed in the following definitions, so we can assume that vector transport $\mathrm{T}$ is isometric, e.g., parallel transport in (\ref{eq:ParallelTransport}). In fact, there are many ways to construct isometric vector transports; see \cite[Section 2.3]{huang2015riemannian}.
\begin{definition}[{\cite[Assumption 3]{huang2015riemannian}}]
	\label{defn:TwiceLipschitzContinuouslyDifferentiable}
	A function $f\colon \mathcal{M} \rightarrow \mathbb{R}$ is \emph{twice Lipschitz continuously differentiable with respect to $\mathrm{T}$ in $\mathcal{U} \subset \mathcal{M}$} if it is twice differentiable and there exists a constant $\kappa>0$ such that, for all $x, y \in \mathcal{U}$, 
	$
	\|\operatorname{Hess} f(y)-\mathrm{T}_{\eta} \operatorname{Hess} f(x) \mathrm{T}_{\eta}^{-1}\| \leq \kappa d(x, y),
	$
	where $\eta=\mathrm{R}_{x}^{-1} y$.
\end{definition}
\begin{lemma}[{\cite[Lemma 4]{huang2015riemannian}}]
	\label{lem:TwiceLipschitzContinuouslyDifferentiable}
	If $f\colon \mathcal{M} \rightarrow \mathbb{R}$ is $ C^3 $, then for any $\bar{x} \in \mathcal{M}$ and any isometric vector transport $\mathrm{T}$, there exists a neighborhood $\mathcal{U}$ of $\bar{x}$ such that $f$ is twice Lipschitz continuously differentiable with respect to $\mathrm{T}$ in $\mathcal{U}$.
\end{lemma}

If the operator, $\operatorname{Hess}f(x)$, above is replaced by a general covariant derivative $\nabla F (x)$, we can get the next results in a similar way.
Lemma \ref{lem:LipCountNabulaF} can be proven in the same way as Lemma \ref{lem:TwiceLipschitzContinuouslyDifferentiable}.
\begin{definition}
	Given a vector field $F$ on $\mathcal{M}$.
	The map $x \mapsto \nabla F(x)$ is \emph{Lipschitz continuous with respect to $\mathrm{T}$ in $\mathcal{U} \subset \mathcal{M}$} if there exists a constant $\kappa>0$ such that, for all $x, y \in \mathcal{U}$, it holds that
	$
	\| \nabla F(y)
	-
	\mathrm{T}_{\eta} \nabla F(x) \mathrm{T}_{\eta}^{-1}
	\| \leq \kappa d(x, y),
	$
	where $\eta=\mathrm{R}_{x}^{-1} y$.
\end{definition}
\begin{lemma}\label{lem:LipCountNabulaF}
	If $F$ is a $C^{2}$ vector field, then
	for any $\bar{x} \in \mathcal{M}$ and any isometric vector transport $\mathrm{T}$, there exists a neighborhood $\mathcal{U}$ of $\bar{x}$ such that the map $x \mapsto \nabla F(x)$ is Lipschitz continuous with respect to $\mathrm{T}$ in $\mathcal{U}$.
\end{lemma}

\subsection{Auxiliary Lemmas}

Notice that in the previous subsection on the definitions of Lipschitz continuity, we used $\|\eta\|$ with $\eta=\mathrm{R}_x^{-1} y$ or $d(x, y)$ to denote the upper bound on the right-hand side. The next lemma shows that the two are not essentially different. When $\mathcal{M}=\mathbb{R}^{n}$, both reduce to $\| x-y \|$.

\begin{lemma}[{\cite[Lemma 2]{huang2015riemannian}}]\label{lem:distR1}
	Let $\mathcal{M}$ be a Riemannian manifold with a retraction $\mathrm{R}$ and let $\bar{x} \in \mathcal{M}$. 
	Then,
	\begin{enumerate}[(i)]
		 \item there exist $a_{0}, a_{1}, \delta_{a_{0}, a_{1}}>0$ such that for all $x$ in a sufficiently small neighborhood of $\bar{x}$ and all $\xi, \eta \in T_{x} \mathcal{M}$ with $\|\xi\|, \|\eta\|  \leq \delta_{a_{0}, a_{1}}$,
		 one has
		 $
		 a_{0}\|\xi-\eta\| \leq d(\mathrm{R}_{x}(\eta), \mathrm{R}_{x}(\xi)) \leq a_{1}\|\xi-\eta\|.
		 $	
		 In particular,
		 $
		 a_{0}\|\xi\| \leq d(x, \mathrm{R}_{x}(\xi)) \leq a_{1}\|\xi\|
		 $
		 when $\eta=0$;
		 
		 \item
		 there exist $a_{0},a_{1}>0$ such that for all $x$ in a sufficiently small neighborhood of $\bar{x}$, one has
		 $
		 a_{0}\|\xi\| \leq d(x,\bar{x}) \leq a_{1}\|\xi\| \text { where } \xi=\mathrm{R}_{\bar{x}}^{-1}(x). 
		 $
	\end{enumerate}
\end{lemma}

The next lemma is the fundamental theorem of calculus in the Riemannian case.
\begin{lemma}[{\cite[Lemma 8]{huang2015riemannian}}]\label{lem:FundamentalTheoremOfCalculusRiemannian}
	Let $F$ be a $C^{1}$ vector field and $\bar{x} \in \mathcal{M}$. 
	Then there exist a neighborhood $\mathcal{U}$ of $\bar{x}$ and a constant $c_{1} \geq 0$ such that for all $x, y \in \mathcal{U}$,
	\begin{equation*}
	\left\| \mathrm{P}_{\gamma}^{0 \to 1} [F(y)]-F(x)-\int_{0}^{1} \mathrm{P}_{\gamma}^{0 \to t} \nabla F(\gamma(t)) \mathrm{P}_{\gamma}^{t \to 0} [\eta] \mathrm{d} t \right\| 
	\leq c_{1} \|\eta\|^{2},
	\end{equation*}
	where $\eta=\mathrm{R}_{x}^{-1}(y)$ and $\mathrm{P}_{\gamma}$ is the parallel transport along the curve $\gamma(t):=\mathrm{R}_{x}\left(t \eta\right)$.
	Moreover, if $\mathrm{R}=\operatorname{Exp}$, then indeed $c_1=0$ above (see \cite[equation (2.4)]{ferreira2002kantorovich}).
\end{lemma}

The next lemma is a Riemannian extension of some important estimates, usually used to analyze Newton methods \cite[Lemma 4.1.12]{dennis1996numerical}.
\begin{lemma}\label{lem:R1est}
	Let $F$ be a $ C^2 $ vector field and $\bar{x}\in\mathcal{M}$. Then there exist a neighborhood $\mathcal{U}$ of $\bar{x}$ and a constant $c_{2}>0$ such that for all $x \in \mathcal{U}$,
	\begin{equation*}
	\left\|\mathrm{P}_{\gamma}^{0 \to 1} [F(x)]-F(\bar{x})-\nabla F(\bar{x}) [\eta] 
	\right\|
	 \leq c_{2} d^{2}(\bar{x} , x),
	\end{equation*}
	where $\eta=\mathrm{R} _{\bar{x}}^{-1} x$ and $\mathrm{P}_{\gamma}$ is the parallel transport along the curve $\gamma(t):=\mathrm{R}_{\bar{x}}(t \eta)$.
\end{lemma}
\begin{proof} 
Let $LHS:=\|\mathrm{P}_{\gamma}^{0 \to 1} [F(x)]-F(\bar{x})-\nabla F(\bar{x}) [\eta] \|.$ It follows that
\begin{equation*}
		\begin{aligned}
			LHS
			& \leq
			\left\|\mathrm{P}_{\gamma}^{0 \to 1} [F(x)]-F(\bar{x})
			-\int_{0}^{1} \mathrm{P}_{\gamma}^{0 \to t} \nabla F(\gamma(t)) \mathrm{P}_{\gamma}^{t \to 0} [\eta] \mathrm{d} t  \right\| \\
			& \quad +
			\left\|\int_{0}^{1} \mathrm{P}_{\gamma}^{0 \to t} \nabla F(\gamma(t)) \mathrm{P}_{\gamma}^{t \to 0} [\eta] \mathrm{d} t -\nabla F(\bar{x}) [\eta] \right\| \\
			&\leq 
			c_{1} \|\eta\|^{2} +
			\Big\|
			\int_{0}^{1}\left(\mathrm{P}_{\gamma}^{0 \to t} \nabla F(\gamma(t)) \mathrm{P}_{\gamma}^{t \to 0} -\nabla F(\bar{x})\right)
			[\eta] \mathrm{d} t\Big\|.
			\text{ (by Lemma \ref{lem:FundamentalTheoremOfCalculusRiemannian})}
		\end{aligned}
\end{equation*}
Let
$\theta := \left\| \int_{0}^{1}\left(\mathrm{P}_{\gamma}^{0 \to t} \nabla F(\gamma(t)) \mathrm{P}_{\gamma}^{t \to 0} -\nabla F(\bar{x})\right)
[\eta] \mathrm{d} t \right\|$.
Note that
\begin{align*}
	\theta 
	&\leq 
	\int_{0}^{1} 
	\left\|
	\mathrm{P}_{\gamma}^{0 \to t} \nabla F(\gamma(t)) \mathrm{P}_{\gamma}^{t \to 0} -\nabla F(\bar{x})
	\right\|
	\left\| \eta \right\|
	\mathrm{d} t \\
	&\leq  
	\int_{0}^{1} 
	c_{0} d(\bar{x}, \mathrm{R}_{\bar{x}}(t \eta))
	\left\| \eta \right\|
	\mathrm{d} t	\text{ (by Lemma \ref{lem:LipCountNabulaF})} \\
	&\leq   
	\int_{0}^{1} 
	c_{0} a_{1}t \|\eta\|
	\left\| \eta \right\|
	\mathrm{d} t 	
	= 
	\frac{1}{2} 
	c_{0} a_{1} \|\eta\|^{2}. 	
	\text{ (by (i) of Lemma \ref{lem:distR1})}
\end{align*}
Combining the above results yields
$$
LHS
\leq (c_{1}+ \frac{1}{2} 
c_{0} a_{1})
\|\eta\|^{2}
\leq  
(c_{1}+ \frac{1}{2} 
c_{0} a_{1})  / a_{0}^{2}
d^{2}(\bar{x} , x),
$$
where the last inequality comes from (ii) of Lemma \ref{lem:distR1}. Letting $c_{2}:= (c_{1}+ \frac{1}{2} 
c_{0} a_{1})  / a_{0}^{2} $ completes the proof.
\end{proof}

We end this section with the following useful lemmas.
\begin{lemma}[{\cite[Lemma 3.2]{fernandes2017superlinear}}]\label{lem:nonsingular}
	Given a vector field $F$ on $\mathcal{M}$.
	If the map $p \mapsto \nabla F(p)$ is continuous at $p^{*}$ and $\nabla F(p^{*})$ is nonsingular, 
	then there exist a neighborhood $\mathcal{U}$ of $p^{*}$ and a constant $\Xi > 0$ such that, for all $p \in \mathcal{U}$, $\nabla F(p)$ is nonsingular and $\left\|\nabla F(p)^{-1}\right\| \leq \Xi$.
\end{lemma}
\begin{lemma}[{\cite[Lemma 14.5]{gallivan2012riemannian}}]\label{lem:dist_zero}
	Let $F$ be a  $C^2$ vector field on $\mathcal{M}$ and $p^{*} \in \mathcal{M}$. If $ F(p^{*})=0 $ and $\nabla F(p^{*}) $ is nonsingular, then there exist a neighborhood $\mathcal{U}$ of $p^{*}$ and constants $c_{3}, c_{4} > 0$ such that, for all $p \in \mathcal{U}$, $ c_{3} d(p, p^{*}) \leq \|F(p)\| \leq c_{4} d(p, p^{*}).$
\end{lemma}
\begin{lemma}[{\cite[Lemma 3.5]{do1992riemannian}}]\label{lem:gauss}
	Let $u \in T_{p} \mathcal{M}$ such that $\operatorname{Exp}_{p}(u)$ exists and $v \in T_p \mathcal{M} \cong T_p \left(T_p \mathcal{M}\right)$.
	Then
	$\langle
	\mathcal{D} \operatorname{Exp}_{p}(u)[u],
	\mathcal{D} \operatorname{Exp}_{p}(u)[v]
	\rangle
	=\langle u, v\rangle.$
	In particular, $\left\|\mathcal{D} \operatorname{Exp}_{p} (\lambda u)[u]\right\|=\left\|u\right\|$ holds all $\lambda \geq 0.$
\end{lemma}

\section{Local Convergence}\label{sec:local}
Here, for any two nonnegative sequences $\left\{u_{k}\right\}$ and $\left\{v_{k}\right\}$, we write $u_{k}=O(v_{k})$ if there is a constant $M >0$ such that
$u_{k}\leq M v_{k}$
for all sufficiently large $k$; and we write
$u_{k}=o(v_{k})$
if $v_{k}>0$ and the sequence of ratios $\{u_{k} / v_{k}\} $ approaches zero.
In this section, we will establish local convergence of our prototype Algorithm \ref{algo:prototype_algo} of RIPM. 

\subsection{Perturbed Damped Riemannian Newton Method}
We will rely on an application of the so-called \emph{perturbed damped} Riemannian Newton method for solving the singularity problem (\ref{eq:Singularity}), which can be stated as Algorithm \ref{algo:PerturbedDampedNewtonMethod}.

\begin{algorithm}
	\caption{Perturbed Damped Riemannian Newton Method for (\ref{eq:Singularity})}
	\label{algo:PerturbedDampedNewtonMethod}
	\SetAlgoLined
	
	\KwIn{A vector field $F \in \mathfrak{X}(\mathcal{M})$, an initial point $p_0 \in \mathcal{M}$ and a retraction $\mathrm{R}$ on $\mathcal{M}$. Set $\mu_0>0$.}
	\KwOut{Sequence $\{p_k\} \subset \mathcal{M}$ such that $\left\{p_k\right\} \rightarrow p^*$ and $F(p^{*})=0_{p^{*}}$.}
	
	Set $k \to 0$\;
	
	\While{Stopping criterion not satisfied}{
		1. Obtain $\xi_k \in T_{p_k}\mathcal{M}$ by solving the perturbed Newton equation:
		\begin{equation}
			\nabla F(p_{k}) [\xi_k] = - F(p_{k}) + \mu_{k} \hat{e};
		\end{equation}
	
		2. Choose a (damped) step size $ 0<\alpha_{k} \leq 1$\;
		
		3. Compute the next point as $p_{k+1}:=\mathrm{R}_{p_{k}} (\alpha_{k} \xi_k)$\;
		
		4. Choose $0<\mu_{k+1}<\mu_k$\;
		
		5. $k \to k + 1$\;
	}
\end{algorithm}

In contrast to the standard Riemannian Newton method described in Algorithm \ref{algo:StandardNewton}, the term ``perturbed'' means that we solve a Newton equation with a perturbed term $\mu_{k} \hat{e}$, while ``damped'' means using $\alpha_k$ instead of unit steps.
It is well known that Algorithm \ref{algo:StandardNewton} are locally superlinearly \cite{fernandes2017superlinear} and quadratically \cite{ferreira2002kantorovich} convergent under the following Riemannian Newton assumptions:
\setlist[assumptions,1]{label=\textbf{(B\arabic*)}}
\begin{assumptions}[leftmargin=9mm]
	\item There exists $p^{*} \in \mathcal{M}$ such that $F(p^{*})=0_{p^{*}}$;
	\label{A1}
	\item The covariant derivative $\nabla F (p^{*})$ is nonsingular;
	\label{A2}
	\item The vector field $F$ is $C^2$.
	\label{A3}
\end{assumptions}

As the following Proposition \ref{prop:PD_Newton} shows, Algorithm \ref{algo:PerturbedDampedNewtonMethod} also has the same convergence properties as Algorithm \ref{algo:StandardNewton} if we control $\mu_{k}$ and $\alpha_{k}$ according to the two schemes that Proposition \ref{prop:PD_Newton} gives.
We can see that either scheme will have $\mu_k \to 0$ and $\alpha_k \to 1$, which makes Algorithm \ref{algo:PerturbedDampedNewtonMethod} eventually reduce to Algorithm \ref{algo:StandardNewton} when $k$ is sufficiently large.

\begin{proposition}[Local convergence of Algorithm \ref{algo:PerturbedDampedNewtonMethod}]\label{prop:PD_Newton}
	Consider the perturbed damped Riemannian Newton method
	described in Algorithm \ref{algo:PerturbedDampedNewtonMethod} for the singularity problem (\ref{eq:Singularity}). 
	Let \ref{A1}-\ref{A3} hold.
	Choose parameters $ \mu_{k}, \alpha_{k}$ as follows; then there exists a constant $\delta>0$ such that for all $p_{0}\in \mathcal{M}$ with
	$
	d(p_{0} , p^{*})<\delta,
	$
	the sequence $\left\{p_{k}\right\}$ is well defined.
	Furthermore,
	\begin{enumerate}[(1)]
		\item if we choose $\mu_{k}=o(\|F(p_{k})\|)$ and $\alpha_{k} \to 1$, then $p_{k} \to p^{*}$ superlinearly;		
		\item if we choose $\mu_{k}=O(\|F(p_{k})\|^{2})$ and $1-\alpha_{k}=O(\|F(p_{k})\|)$, then $p_{k} \to p^{*}$ quadratically.
	\end{enumerate}
\end{proposition}
\begin{proof}
By \ref{A2}-\ref{A3}, Lemma \ref{lem:LipCountNabulaF} and Lemma \ref{lem:nonsingular}, we can let $p_{k} $ be sufficiently close to $p^{*}$ such that $\nabla F(p_{k})$ is nonsingular, and
$\left\|\nabla F(p_{k})^{-1}\right\| \leq \Xi$ for some constant $\Xi$.
Then, the next iterate point,
\begin{equation*}
p_{k+1}:=\mathrm{R}_{p_{k}}[\alpha_{k} \nabla F(p_{k})^{-1}(-F(p_{k})+\mu_{k} \hat{e})],
\end{equation*}
is well defined in Algorithm \ref{algo:PerturbedDampedNewtonMethod}, and it follows from $ p^{*}=\mathrm{R}_{p_{k}}(\eta) $ with $ \eta := \mathrm{R}_{p_{k}}^{-1} p^{*} $ and (i) of Lemma \ref{lem:distR1} that
\begin{align}
	d(p_{k+1} , p^{*})
	&\leq
	a_{1}
	\| \eta - \alpha_{k} \nabla F(p_{k})^{-1}(-F(p_{k})+\mu_{k} \hat{e}) \| \notag \\
	&=
	a_{1}  
	\|\eta + \alpha_{k} \nabla F(p_{k})^{-1}(F(p_{k})-\mu_{k} \hat{e})\|. \label{eq:1114}
\end{align}
Let 
$r_{k} := \eta + \alpha_{k} \nabla F(p_{k})^{-1}(F(p_{k})-\mu_{k} \hat{e})$.
Algebraic manipulations
show that
\begin{equation*}
r_{k}=
(1-\alpha_{k}) \eta  +
\alpha_{k} \nabla F(p_{k})^{-1} 
[\nabla F(p_{k}) \eta + F(p_{k})-\mathrm{P}_{\gamma}^{0 \to 1} F(p^{*})-\mu_{k} \hat{e}],
\end{equation*}
where $\mathrm{P}_{\gamma}$ is the parallel transport along the curve $\gamma(t)=\mathrm{R}_{p_{k}}(t \eta)$ and $F(p^{*})=0.$
Thus, using $\left\|\eta\right\| \leq d(p_{k}, p^{*})/a_{0}$ from (ii) of Lemma \ref{lem:distR1} and Lemma \ref{lem:R1est}, we have
\begin{align*}
	\| r_{k}\|&\leq
	(1-\alpha_{k}) \| \eta\| \notag 
	+ 
	\alpha_{k} \|\nabla F(p_{k})^{-1} \|
	\|\mathrm{P}_{\gamma}^{0 \to 1} F(p^{*}) - F(p_{k}) -
	\nabla F(p_{k}) \eta \| \\
	& \quad + \alpha_{k} \|\nabla F(p_{k})^{-1} \| \|\hat{e}\| \mu_{k} \\ 
	&\leq (1-\alpha_{k})d(p_{k} , p^{*}) /a_{0} +
	\alpha_{k} \|\nabla F(p_{k})^{-1} \| 
	c_{2} d^{2}(p_{k} , p^{*})\\
	& \quad
	+ \alpha_{k}\|\nabla F(p_{k})^{-1} \|  \|\hat{e}\| \mu_{k}\\
	&\leq (1-\alpha_{k})d(p_{k} , p^{*}) /a_{0} +
	\Xi 
	c_{2} d^{2}(p_{k} , p^{*})
	+ \Xi  \|\hat{e}\| \mu_{k}.
\end{align*}
Combining the above with (\ref{eq:1114}), we conclude that
\begin{equation}\label{77}
	d(p_{k+1} , p^{*}) \leq \kappa_{1}(1-\alpha_{k}) d(p_{k} , p^{*}) 
	+ \kappa_{2} d^{2}(p_{k} , p^{*})
	+ \kappa_{3} \mu_{k}
\end{equation}
for some positive constants $ \kappa_{1},\kappa_{2},\kappa_{3}$.
On the other hand, by Lemma \ref{lem:dist_zero}, we have
\begin{equation}\label{eq:LipFbigoh}
	\left\|F(p_{k})\right\|=O(d(p_{k}, p^{*})).
\end{equation}
In what follows, we prove assertions (1) and (2).

\textbf{(1)} 
Suppose that $\alpha_{k} \to 1$ and $\mu_{k}=o(\|F(p_{k})\|)$, which together with (\ref{eq:LipFbigoh}) imply $\mu_{k}=o(d(p_{k}, p^{*}))$.
By (\ref{77}), we have
\begin{equation}\label{eq:superlinear}
	\frac{d(p_{k+1} , p^{*})}{d(p_{k} , p^{*})}\leq
	\kappa_{1} (1-\alpha_{k})
	+\kappa_{2} d(p_{k} , p^{*})
	+\kappa_{3}\frac{\mu_{k}}{d(p_{k} , p^{*})}.
\end{equation}
We can take $\delta$ sufficiently small and $ k $ sufficiently large, if necessary, to conclude that
$
d(p_{k+1} , p^{*})<\frac{1}{2}d(p_{k} , p^{*}) <\delta.
$
Thus, $p_{k+1} \in B_{\delta}(p^{*})$, the open ball of radius $\delta$ centered at $p^{*}$ on $\mathcal{M}$. 
By induction, it is easy to show that the sequence $\left\{p_{k}\right\}$ is well defined and converges to $p^{*}$. 
Taking the limit of both sides of (\ref{eq:superlinear}) proves superlinear convergence.

\textbf{(2)} 
Again, we start from (\ref{77}) and rewrite it as:
\begin{align}\label{eq:bigoh}
	d(p_{k+1} , p^{*}) = (1-\alpha_{k}) O(d(p_{k} , p^{*}) )
	+O(d^{2}(p_{k} , p^{*}) )
	+O(\mu_{k}).
\end{align}
Suppose that $1-\alpha_{k}=O(\left\|F(p_{k})\right\|)$
and $\mu_{k}=O(\left\|F(p_{k})\right\|^{2})$.
Using (\ref{eq:LipFbigoh}), the above reduces to
$
d(p_{k+1} , p^{*}) = O(d^{2}(p_{k} , p^{*})).
$
This implies that there exists a constant $\nu$ such that $d(p_{k+1} , p^{*}) \leq \nu d^{2}(p_{k} , p^{*})$, and hence,
$
d(p_{k+1} , p^{*}) \leq \nu d^{2}(p_{k} , p^{*}) \leq \nu \delta^{2}<\delta,
$
if $\delta$ is sufficiently small.
Again, by induction, $\left\{p_{k}\right\}$ converges to $p^{*}$ quadratically.
\end{proof}

\subsection{Local Convergence of Algorithm \ref{algo:prototype_algo}}
Next, lemma shows the relationship between the parameter $ \gamma_{k} $ and step size $ \alpha_{k} $ in Algorithm \ref{algo:prototype_algo}. 
\begin{lemma}\label{lem:gamma_alpha}
	Consider the Algorithm \ref{algo:prototype_algo} for solving the problem (\ref{RCOP}).  
	Let \ref{B1} and \ref{B4} hold at some $w^{*}=(x^{*}, y^{*}, z^{*}, s^{*})$ and $ \alpha_{k} $ be as in (\ref{rule:StepSizeA}). 
	Define a constant,
	\begin{equation*}
	\Pi := 2 \max \left\{\max _{i}\left\{
	1/(s^{*})_{i}
	\mid (s^{*})_{i}>0\right\}, \max _{i}\left\{
	1/(z^{*})_{i}
	\mid (z^{*})_{i}>0\right\}\right\}.
	\end{equation*}
	For $\gamma_{k} \in(0,1)$, if
	$
	\Pi \left\| \Delta w_{k}\right\| \leq \gamma_{k},
	$
	then
	$
	0  \leq 1-\alpha_{k} \leq
	(1-\gamma_{k})+\Pi \left\| \Delta w_{k}\right\|.
	$
\end{lemma}
\begin{proof} 
Notice that the fourth line of (\ref{eq:FullNewtonEquation}) yields
$$
S_{k}^{-1} \Delta s_{k}+Z_{k}^{-1} \Delta z_{k}=\mu_{k}(S_{k} Z_{k})^{-1} e-e,
$$
which is exactly the same as in the usual interior point method in the Euclidean setting. 
Thus, the proof entails directly applying \cite[Lemma 3 and 4]{yamashita1996superlinear} for the Euclidean case to the Riemannian case.
\end{proof}

Now, let us establish the local convergence of our Algorithm \ref{algo:prototype_algo} in a way that replicates Proposition \ref{prop:PD_Newton} except for taking account of parameter $\gamma_{k} $.
\begin{theorem}[Local convergence of prototype Algorithm \ref{algo:prototype_algo}]\label{thm:LocalConvergence}
	Consider the Algorithm \ref{algo:prototype_algo} for solving the problem (\ref{RCOP}).  
	Let \ref{B1}-\ref{B5} hold at some $w^{*}$.
	Choose parameters $\mu_k, \gamma_k$ as follows;
	then there exists a constant $\delta>0$ such that, for all
	$
	w_{0} \in\mathcal{N}
	$
	with $d(w_{0} , w^{*})<\delta,$
	the sequence $\left\{w_{k}\right\}$ is well defined. Furthermore,
	\begin{enumerate}[(1)]
		\item if we choose $\mu_{k}=o(\left\|F(w_{k})\right\|)$ and $\gamma_{k} \to 1$, then $w_k \to w^*$ superlinearly;	
		\item if we choose $\mu_{k}=O(\left\|F(w_{k})\right\|^{2})$
		and $1-\gamma_{k}=O(\left\|F(w_{k})\right\|)$, then $w_k \to w^*$ quadratically.
	\end{enumerate}
\end{theorem}
\begin{proof}
We only prove (2) because (1) can be proven in the same way.
Let $w_{k} $ be such that  $d(w_{k} , w^{*})<\delta$ for sufficiently small $\delta$.
From Proposition \ref{prop:non}, \ref{B1}-\ref{B5} shows that the KKT vector field $F$ satisfies \ref{A1}-\ref{A3}; thus, the discussion in the proof of Proposition \ref{prop:PD_Newton} applies to KKT vector field $F$ as well, simply by replacing the symbol $p$ with $w$.
Since we choose $\mu_{k}=O(\|F(w_{k})\|^{2})$, and $\|F(w_{k})\|=O(d(w_{k}, w^{*}))$ by (\ref{eq:LipFbigoh}), we obtain
$
\mu_{k}=O(d^{2}(w_{k}, w^{*})).
$
Thus,
\begin{align*}
	 \left\|\Delta w_{k}\right\|  & =
	\left\|
	\nabla F(w_{k})^{-1} (-F(w_{k})+\mu_{k} \hat{e})
	\right\| \text{ (by (\ref{eq:prototype}))}\\
	& \leq 
	\Xi
	(\left\|F(w_{k})\right\|+\mu_{k}\|\hat{e}\|) \text{ (by Lemma \ref{lem:nonsingular})}\\
	& \leq  O(\|F(w_k)\|)+O(\mu_{k}) \\
	& =O(d(w_{k} , w^{*}))+O(d^{2}(w_{k} , w^{*}) )
	=O(d(w_{k} , w^{*})).
\end{align*}
Since $\delta$ is sufficiently small, from the above inequalities, the condition of Lemma \ref{lem:gamma_alpha} is satisfied, i.e.,
$
\left\|\Delta w_{k}\right\| \leq \hat{\gamma} / \Pi \leq \gamma_k / \Pi
$
for a constant $\hat{\gamma} \in (0,1).$
Hence, by $1-\gamma_k=O(\left\|F\left(w_k\right)\right\|)$, one has
$
1-\alpha_{k} \leq
(1-\gamma_{k})+\Pi\left\|\Delta w_{k}\right\| 
= (1-\gamma_{k})+O(d(w_{k} , w^{*}))
= O(d(w_{k} , w^{*})).
$
Finally, from (\ref{eq:bigoh}),
we have
$
d(w_{k+1} , w^{*}) = (1-\alpha_{k}) O(d(w_{k} , w^{*}) )
+O(d^{2}(w_{k} , w^{*}) )
+O(\mu_{k}) 
=O(d^{2}(w_{k} , w^{*})).
$
This completes the proof.
\end{proof}

Algorithm \ref{algo:prototype_algo} guaranteed the local convergence, but we are more interested in its globally convergent version (described in the next section). We still provide a simple example of Algorithm \ref{algo:prototype_algo} online
\footnote{See .../LocalRIPM/PrototypeRIPM.m in \url{https://doi.org/10.5281/zenodo.10612799}}, where from (2) of Theorem \ref{thm:LocalConvergence}, in practice we set parameters 
$\mu_{k+1}=\min\{\mu_{k}/1.5,0.5\left\|F(w_k)\right\|^2\};
$ and 
$\gamma_{k+1}=\max\{\hat{\gamma}, 1- \left\|F(w_k)\right\|\}
$
with $\hat{\gamma}=0.5$.

\section{Global Algorithm}\label{sec:global}

The globally convergent version of our RIPM uses the classical line search described in \cite{el1996formulation}. 
The following considerations and definitions are needed in order to describe it compactly.
For simplicity, we often omit the subscript of iteration count $k$.

Given the current point $w=(x, y, z, s)$ and $\Delta w=(\Delta x, \Delta y, \Delta z, \Delta s)$, the next iterate is obtained along a curve on product manifold $\mathcal{N}$, i.e.,
$
\alpha \mapsto w(\alpha):=\bar{\mathrm{R}}_{w}(\alpha \Delta w)
$
with some step size $\alpha >0$, see (\ref{eq:barR}) for $\bar{\mathrm{R}}_{w}$.
By introducing
$w(\alpha)=(x(\alpha), y(\alpha), z(\alpha), s(\alpha))$, 
we have $x(\alpha)=\mathrm{R}_{x}(\alpha \Delta x)$, $y(\alpha)=y+\alpha  \Delta y$, $z(\alpha)=z+\alpha  \Delta z$, and $s(\alpha)=s+\alpha  \Delta s$.
For a given starting point $w_{0} \in \mathcal{N}$ with $(z_{0}, s_{0})>0$, let us set two constants
\begin{equation*}
	\tau_{1}:=\min (Z_{0} S_{0} e)/ (z_{0}^{T} s_{0} / m), 
	\; 
	\tau_{2}:=z_{0}^{T} s_{0}/\left\|F(w_{0})\right\|.
\end{equation*}
As well, define two functions
$
f^{I}(\alpha):=\min (Z(\alpha) S(\alpha)e)-\gamma \tau_{1} z(\alpha)^{T} s(\alpha) / m$,
$
f^{II}(\alpha):=z(\alpha)^{T} s(\alpha)-\gamma \tau_{2}\|F(w(\alpha))\|,
$
where $\gamma \in (0,1)$ is a constant.
For $i=I, I I$, define
\begin{equation}\label{step}
	\alpha^{i}:=\max _{\alpha \in (0,1]}\left\{\alpha: f^{i}(t) \geq 0, \text { for all } t \in (0, \alpha] \right\},
\end{equation}
i.e., $\alpha^i$ are either one or the smallest positive root for the functions $f^i(\alpha)$ in $(0,1]$.

Moreover, we define the merit function $\varphi\colon \mathcal{N} \to \mathbb{R}$ by $\varphi(w):=\|F(w)\|^{2}$; accordingly, we have $\operatorname{grad} \varphi(w)=2\nabla F(w)^{*} [F(w)],$ where symbol $*$ means its adjoint operator.
Let $\| \cdot \|_{1}$, $\| \cdot \|_{2}$ be $l_{1}, l_{2}$ vector norms.
Note that
$\|F(w)\|_{w}^{2}=\left\|\operatorname{grad}_{x} \mathcal{L}(w)\right\|_{x}^{2}+\|h(x)\|_{2}^{2}+\|g(x)+s\|_{2}^{2}+\|ZSe\|_{2}^{2}$ by (\ref{eq:KKTVectorField}).
Moreover, for any nonnegative $z,s \in \mathbb{R}^{m}$, one has
$
\left\|Z Se\right\|_{2} \leq z^{T} s = \|Z Se\|_{1}\leq \sqrt{m} \left\|Z Se\right\|_{2}.
$
Hence,
\begin{equation}\label{genlma2}
	\left\|Z Se\right\|_{2} / \sqrt{m} 
	\leq z^{T} s / \sqrt{m}
	\leq \left\|Z Se\right\|_{2} 
	\leq \|F(w)\|.
\end{equation}
Now, using the above definitions, the globally convergent RIPM can be stated as Algorithm \ref{algo:GlobalRIPM}.

\begin{algorithm}
	\caption{Global Convergent Algorithm of RIPM for (\ref{RCOP})}
	\label{algo:GlobalRIPM}
	\SetAlgoLined
	
	\KwIn{An initial point $w_0=\left(x_0, y_0, z_0, s_0\right) \in \mathcal{N}$ with $\left(z_0, s_0\right)>0$ and a retraction $\mathrm{R}$ on $\mathcal{M}$. $ \theta \in(0,1), \beta \in(0,0.5], \gamma_{-1} \in (0.5,1).$}
	
	\KwOut{Sequence $\left\{w_k\right\} \subset \mathcal{N}$ such that $\left\{w_k\right\} \rightarrow w^*$ and $w^*$ satisfies the KKT conditions (\ref{eq:KKTConditions}).}
	
	Set $k \to 0$\;
	
	\While{Stopping criterion not satisfied}{
		1. Set $\sigma_{k} \in(0,1), \rho_{k} \in [z_{k}^{T} s_{k} / m, \left\|F(w_{k})\right\| / \sqrt{m} ]$\;
		2. Obtain $\Delta w_k=\left(\Delta x_k, \Delta y_k, \Delta z_k, \Delta s_k\right) \in T_{w_k} \mathcal{N}$ by solving the following linear equation:
		\begin{equation}\label{eq:NT}
			\nabla F(w_{k}) [\Delta w_{k}]=-F(w_{k})+\sigma_{k} \rho_{k} \hat{e};
		\end{equation}
		
		3. Step size selection:
		
		(3a) {Centrality condition}:
		Set $\gamma_{k} \in (0.5, \gamma_{k-1})$ and $\bar{\alpha}_{k}=\min \{\alpha_{k}^{I}, \alpha_{k}^{II}\}$ from (\ref{step})\;

			(3b) {Sufficient decrease condition}:
		Let $\alpha_{k}:=\bar{\alpha}_{k}$\;	
		\While{$\alpha_{k}$ does not satisfy the condition $\varphi(\bar{\mathrm{R}}_{w_{k}}(\alpha_{k} \Delta w_{k}))-\varphi(w_{k}) \leq \alpha_{k}  \beta \left\langle\operatorname{grad} \varphi_{k}, \Delta w_{k}\right\rangle$}{
			$\alpha_{k} := \theta \alpha_{k}$\;
		}
				
		4. Compute the next point as $w_{k+1}:=\bar{\mathrm{R}}_{w_{k}}(\alpha_{k} \Delta w_{k})$\;
		5. $k \to k+1$\;
	}
\end{algorithm}

Compared to the prototype Algorithm \ref{algo:prototype_algo}, the global Algorithm \ref{algo:GlobalRIPM} involves a more elaborate choice of step size.
Regarding the centrality condition (3a), it does not differ in any way from the Euclidean setting; thus, references \cite{el1996formulation,bonettini2005inexact,durazzi2000newton} show that $\bar{\alpha}_k$ is well-defined thereby ensuring that $z_k$, $s_k$ are positive.
In the following, we focus on the sufficient decrease (Armijo) condition (3b):
\begin{equation}\label{eq:ArmijoCondition}
	\varphi(\bar{\mathrm{R}}_{w_{k}}(\alpha_{k} \Delta w_{k}))-\varphi(w_{k}) \leq \alpha_{k}  \beta \left\langle\operatorname{grad} \varphi_{k}, \Delta w_{k}\right\rangle,
\end{equation}
where $\operatorname{grad} \varphi_k \equiv \operatorname{grad} \varphi\left(w_k\right)$ for short.
With a slight abuse of notation $\varphi$, at the current point $w$ and direction $\Delta w$, we define a real-to-real function $\alpha \mapsto \varphi(\alpha):=\varphi\left(\bar{\mathrm{R}}_w(\alpha \Delta w)\right)$; then, it follows from the definition of a retraction and the chain rule that
\begin{equation*}
\varphi^{\prime}(0) =\mathcal{D} \varphi(\bar{\mathrm{R}}_{w}(0))\left[\mathcal{D}\bar{\mathrm{R}}_{w}(0)[\Delta w]\right] =\mathcal{D} \varphi(w)[\Delta w] =\langle\operatorname{grad} \varphi(w), \Delta w\rangle. 
\end{equation*}
Hence, 
$\varphi_{k}^{\prime}(0)=\langle \operatorname{grad} \varphi_{k}, \Delta w_{k}\rangle$
at the $ k $-th iterate $w_{k}$;
and then the Armijo condition (\ref{eq:ArmijoCondition}) reads $ \varphi_{k}(\alpha_{k}) - \varphi_{k}(0) \leq \alpha_{k} \beta \varphi_{k}^{\prime}(0)$ as usual.
If $ \varphi_{k}^{\prime}(0)<0 $, the backtracking loop in (3b) of Algorithm \ref{algo:GlobalRIPM} will eventually stop \cite[Lemma 3.1]{nocedal2006numerical}. 
The next lemma shows the condition under which the
Newton direction $\Delta w_{k}$ generated by (\ref{eq:NT}) ensures the descent of the merit function.
\begin{lemma}\label{lem:dr}
	If the direction $\Delta w_{k}$ is the solution of equation (\ref{eq:NT}), then
	\begin{equation*}
		\langle\operatorname{grad} \varphi(w_{k}), \Delta w_{k}\rangle = 2(-\|F(w_{k})\|^{2}+\sigma_{k} \rho_{k} z_{k}^{T} s_{k}).
	\end{equation*}
	In this case, $\Delta w_{k}$ is a descent direction for $\varphi(w)$ at $w_{k}$ if and only if 
	$
	\rho_{k}<\|F(w_{k})\|^{2} / \sigma_{k} z_{k}^{T} s_{k}.
	$
\end{lemma}
\begin{proof}
The iteration count $k$ is omitted.
Let $\Delta w$ be given by (\ref{eq:NT}).
Then, we have
$
\langle\operatorname{grad} \varphi(w), \Delta w\rangle =\left\langle 2 \nabla F(w)^{*} [F(w)], \Delta w\right\rangle 
=2\langle F(w), \nabla F(w) [\Delta w] \rangle
=2\langle F(w), -F(w)+\sigma \rho \hat{e}\rangle
=2(-\langle F(w), F(w)\rangle+\sigma \rho\langle F(w), \hat{e}\rangle).
$
Then, by definition of $\hat{e}$ in (\ref{eq:perturbedKKTvectorfield}),
$ \langle F(w), \hat{e}\rangle 
=\langle ZSe, e\rangle
= z^{T} s$ completes the proof.
\end{proof}

The next proposition shows that Algorithm \ref{algo:GlobalRIPM} can generate the monotonically nonincreasing sequence $\left\{\varphi_{k}\right\}$. 
Note that $\varphi_{k+1} \equiv \varphi_k\left(\alpha_k\right)$ and $\varphi_k \equiv \varphi_k(0) \equiv \varphi\left(w_k\right)$.
\begin{proposition}\label{prop:monotone}
	If $\|F(w_{k})\| \ne 0$, then
	the direction $\Delta w_{k}$ generated by Algorithm \ref{algo:GlobalRIPM} is a descent direction for $\varphi(w)$ at $w_k$. 
	Moreover, if the Armijo condition (\ref{eq:ArmijoCondition}) is satisfied, then
	\begin{equation*}
		\varphi_{k}(\alpha_{k}) \leq\left[1-2 \alpha_{k} \beta(1-\sigma_{k})\right] \varphi_{k}(0) .
	\end{equation*}
	Thus, the sequence $\left\{\varphi_{k}\right\}$ is monotonically nonincreasing.
\end{proposition}
\begin{proof} 
	The iteration count $k$ is omitted. Suppose that $\rho \leq \left\|F(w)\right\| / \sqrt{m}$ and $ \Delta w$ is given by (\ref{eq:NT}), we have
	\begin{align}
		\varphi^{\prime}(0)& =\langle \operatorname{grad} \varphi(w), \Delta w\rangle
		=2(-\varphi(w)+\sigma \rho  z^{T} s) \text{ (by Lemma \ref{lem:dr})} \notag \\
		&\leq 2(-\varphi(w)+\sigma \left\|F(w)\right\|  z^{T} s / \sqrt{m}) \notag \\
		& \leq 2(-\varphi(w)+\sigma  \left\|F(w)\right\|^{2})  \text{ (by (\ref{genlma2}))} \notag \\
		& =-2(1-\sigma) \varphi(w)<0.
		\label{49}
	\end{align}
	Thus, in Algorithm \ref{algo:GlobalRIPM}, $\Delta w$ is a descent direction for the merit function $\varphi$ at $w$.
	Alternatively, by Lemma \ref{lem:dr}, it is sufficient to show that
	$
	\|F(w)\| / \sqrt{m} < \|F(w)\|^{2} /\sigma z^{T} s.
	$
	By $\sigma z^{T} s < z^{T} s \leq \sqrt{m}\|F(w)\|;$
	then,
	$
	1 / \sqrt{m} < \|F(w)\|/\sigma z^{T} s.
	$
	Multiplying both sides by $\|F(w)\|$ yields the result.

	Moreover, if condition (\ref{eq:ArmijoCondition}) is satisfied, then by (\ref{49}), we have
	$
	\varphi(\alpha)  \leq \varphi(0)+\alpha \beta \langle\operatorname{grad} \varphi(w), \Delta w\rangle 
	\leq \varphi(0)+\alpha \beta (-2(1-\sigma) \varphi(0))
	=[1-2 \alpha \beta(1-\sigma)] \varphi(0).
	$
	Note that in Algorithm \ref{algo:GlobalRIPM}, we set $\beta \in(0, 1 / 2], \sigma \in(0,1)$, and $\alpha \in (0,1]$, which imply that the sequence $\left\{\varphi_k\right\}$ is monotonically nonincreasing.
\end{proof}

Finally, we need to make the following assumptions. For $\epsilon \geq 0$, let
$$
\Omega(\epsilon) :=\left\{w\in \mathcal{N} \mid \epsilon \leq \varphi(w) \leq \varphi_{0}, \min (Z S e) /(z^{T} s / m) \geq \tau_{1} / 2, z^{T} s /\|F(w)\| \geq \tau_{2} / 2\right\}.
$$

\setlist[assumptions,1]{label=\textbf{(C\arabic*)}}
\begin{assumptions}[leftmargin=9mm]
	\item \label{C1} 
	In the set $\Omega(0)$, $f, h,$ and $g$ are smooth functions;
	$\left\{\operatorname{grad} h_{i}(x)\right\}_{i=1}^{l}$ is linearly independent for all $x$;	
	and the map $w \mapsto \nabla F (w)$ is Lipschitz continuous (with respect to parallel transport);
	
	\item \label{C2}
	The sequences $\left\{x_{k}\right\}$ and $\left\{z_{k}\right\}$ are bounded \cite{Durazzi2004,bonettini2005inexact};
	
	\item \label{C3}
	In any compact subset of $\Omega(0)$, $\nabla F (w)$ is nonsingular.
\end{assumptions}
Given the above assumptions, we can now prove the following statement.
\begin{theorem}[Global Convergence of RIPM]\label{thm:global}
	Let $\left\{w_{k}\right\}$ be generated by Algorithm \ref{algo:GlobalRIPM} with $\mathrm{R} = \operatorname{Exp}$ and $\left\{\sigma_{k}\right\} \subset(0,1)$ be bounded away from zero and one.
	Let $\varphi$ be Lipschitz continuous on $\Omega(0)$.
	If \ref{C1}-\ref{C3} hold, then $\left\{\left\|F(w_{k})\right\|\right\}$ converges to zero. Moreover, if $w^{*}$ is a limit point of sequence $\left\{w_{k}\right\}$, then $w^{*}$ satisfies Riemannian KKT conditions (\ref{eq:KKTConditions}).
\end{theorem}

The proof of the above theorem will be given in the next section.
Note that although the exponential map is used in the proof, the numerical experiments indicate that global convergence may hold for a general retraction $\mathrm{R}$.

\section{Proof of Global Convergence}\label{sec:ProofGlobal}

In this section, our goal is to prove the global convergence Theorem \ref{thm:global}.
We will follow the proof procedure in \cite{el1996formulation}, which discussed Algorithm \ref{algo:GlobalRIPM} when $\mathcal{M}=\mathbb{R}^{n}$ in (\ref{RCOP}), and we call the algorithm in \cite{el1996formulation} the Euclidean Interior Point Method (EIPM).
In what follows, we will omit similar content because of space limitations and focus on the difficulties encountered when adapting the proof of EIPM to RIPM. 
In particular, we will make these difficulties as tractable as in EIPM by proving a series of propositions in Subsection \ref{subsec:Continuity}.

\subsection{Continuity of Some Special Scalar Fields}\label{subsec:Continuity}

To show the boundedness of the sequences generated by Algorithm \ref{algo:GlobalRIPM}, we need the continuity of some special scalar fields on manifold $\mathcal{M}$. 
The claims of this subsection are trivial if $\mathcal{M}= \mathbb{R}^n$, but they need to be treated carefully for general $\mathcal{M}$.

If we assign a linear operator $\mathcal{A}_{x}\colon T_{x} \mathcal{M} \to T_{x} \mathcal{M}$ to each $x \in \mathcal{M},$
then the map $ x \mapsto  \left\|\mathcal{A}_{x}\right\|:=\sup \left\{\left\|\mathcal{A}_{x}v\right\|_{x} \mid v \in T_{x} \mathcal{M},\|v\|_{x}=1, \text{ or, } \|v\|_{x} \leq 1\right\}$ is a scalar field on $\mathcal{M}$; but notice that the operator norm $ \| \cdot \| $ depends on $ x $.
Let $\operatorname{Sym}(d)$ denote the set of symmetric matrices of order $d$, and $\| \cdot \|_{\mathrm{F}}$, $\| \cdot \|_{2}$ denote the Frobenius norm and the spectral norm, respectively, applied to a given matrix.

\begin{lemma}\label{lem:InvariantOrthonormalBasis}
	Let $\mathcal{M}$ be an $n$-dimensional Riemannian manifold. 
	Let $x \in \mathcal{M}$ and $\mathcal{A}_{x}$ be a linear operator on $T_{x} \mathcal{M}$.
	Choose an orthonormal basis of $T_{x} \mathcal{M}$ with respect to $\langle\cdot,\cdot\rangle_{x}$, and let $\hat{\mathcal{A}}_{x}\in \mathbb{R}^{n \times n}$ denote the matrix representation of $\mathcal{A}_{x}$ under the basis.
	Then, 
	$\|\hat{\mathcal{A}}_{x}\|_{2}$, $\|\hat{\mathcal{A}}_{x}\|_{\mathrm{F}}$ are \emph{invariant} under a change of orthonormal basis; moreover,
	$
	\left\|\mathcal{A}_{x}\right\|=\|\hat{\mathcal{A}}_{x}\|_{2} \leq \|\hat{\mathcal{A}}_{x}\|_{\mathrm{F}}.
	$
\end{lemma}
\begin{proof}
	Suppose that there are two orthonormal bases $\{E_{i}\}_{i=1}^{n},\{E^{\prime}_{i}\}_{i=1}^{n}$ on $T_{x} \mathcal{M}$.
	With respect to them, let $ P \in \mathbb{R}^{n \times n}$ denote the change-of-basis matrix, i.e.,
	$ [P]_{kj} := \langle E^{\prime}_{j},E_{k}\rangle_{x}$, for $ 1\leq k,j \leq n $; then, $ P $ is orthogonal.  
	Let $ \hat{\mathcal{A}}_{x}, \hat{\mathcal{A}}^{\prime}_{x} \in \mathbb{R}^{n \times n}$ denote the matrix representations of $ \mathcal{A}_{x} $ under the two bases, respectively.
	We have
	$\hat{\mathcal{A}}_{x}^{\prime}=P^{-1} \hat{\mathcal{A}}_{x} P.$ Then,
	$
	\|\hat{\mathcal{A}}_{x}^{\prime}\|=\|P^{-1} \hat{\mathcal{A}}_{x} P\|=\|\hat{\mathcal{A}}_{x}\|
	$
	holds for the Frobenius norm or the spectral norm.
	Therefore, the values $\|\hat{\mathcal{A}}_{x}\|_{2}$ and $\|\hat{\mathcal{A}}_{x}\|_{\mathrm{F}}$ are invariant under a change of orthonormal basis.
	Now, consider an orthonormal basis $\{E_{i}\}_{i=1}^{n}$ on $T_{x} \mathcal{M}$.
	For any $ y \in T_{x} \mathcal{M}$, its vector representation $ \hat{y} \in \mathbb{R}^{n}$ is defined by $y=\sum_{i=1}^{n} \hat{y}_{i} E_{i}$.
	Accordingly, we have $\widehat{\mathcal{A}_{x} y}=\hat{\mathcal{A}}_{x} \hat{y}$, i.e.,
	$\mathcal{A}_{x} y=\sum_{i=1}^{n}(\hat{\mathcal{A}}_{x} \hat{y})_{i} E_{i}$, 
	and from the orthonormal property of the basis, 
	we have
	\begin{equation*}
	\left\|\mathcal{A}_{x} y\right\|_{x}^{2}
	=\|\sum_{i=1}^{n}(\hat{\mathcal{A}}_{x} \hat{y})_{i} E_{i}\|_{x}^{2}
	=\sum_{i=1}^{n}(\hat{\mathcal{A}}_{x} \hat{y})_{i}^{2}
	=\|\hat{\mathcal{A}}_{x} \hat{y}\|_{2}^{2}.
	\end{equation*}
	Thus, $\left\|\mathcal{A}_{x} y\right\|_{x}=\|\hat{\mathcal{A}}_{x} \hat{y}\|_{2}$ for any $ y \in T_{x} \mathcal{M}$.
	Finally, we have
	$$
	\left\|\mathcal{A}_{x}\right\|=\sup _{\substack{y \in T_{x} \mathcal{M}, \|y\|_{x}=1}}\left\|\mathcal{A}_{x} y\right\|_{x}
	=\sup _{\substack{\hat{y} \in \mathbb{R}^{n}, \|\hat{y}\|_{2}=1}}\|\hat{\mathcal{A}}_{x} \hat{y}\|_{2}
	=\|\hat{\mathcal{A}}_{x}\|_{2}.
	$$
	It is clear that $ \| X \|_{2} \leq \| X \|_{\mathrm{F}} $ for any matrix $X$.
\end{proof}

Yet, we have not clarified the continuity about $x \mapsto\left\|A_x\right\|$. The following proposition proves the continuity of an important case of $ A_x $ that appears in (\ref{RCOP}).
\begin{proposition}\label{prop:cont_Hessf}
	Consider $f$ in (\ref{RCOP}).
	Let $\hat{\operatorname{Hess}}f(x) \in \operatorname{Sym}(d)$ denote the matrix representation of $\operatorname{Hess}f(x)$ under an arbitrary orthonormal basis of $T_{x} \mathcal{M}$.
	Then,
	$
	x \mapsto \| \hat{\operatorname{Hess}}f(x)\|
	$
	is a continuous scalar field on $\mathcal{M}$ for $\| \cdot \|_{\mathrm{F}}$ or $\| \cdot \|_{2}$.
	Moreover, $x \mapsto \| \operatorname{Hess}f(x)\|$
	is a continuous scalar field on $\mathcal{M}$.
\end{proposition}
\begin{proof}
Lemma \ref{lem:InvariantOrthonormalBasis} shows that $x \mapsto \| \hat{\operatorname{Hess}}f(x) \|$ is well defined, so 
it suffices to prove its continuity. From Corollary 13.8 in \cite{lee2013smooth}, for each $ \bar{x} \in \mathcal{M} $ there is a smooth, orthonormal local frame $\{E_{i}\}_{i=1}^{d}$ on a neighborhood $ \mathcal{U} $ of $ \bar{x} $; namely, $\{E_{1}(x), \ldots,E_{d}(x)\}$ forms an orthonormal basis on $ T_{x} \mathcal{M} $ for all $ x \in \mathcal{U}. $ 
Choose such a local frame $\{E_{i}\}_{i=1}^{d}$ around $\bar{x}$; then, the matrix representation of $\operatorname{Hess}f(x)$ is given by
$
[\hat{\operatorname{Hess}}f(x)]_{kj}
:=
\left\langle\operatorname{Hess}f(x)[E_{j}(x)] , E_{k}(x)\right\rangle_{x} 
=
\langle(\nabla_{E_{j}} \operatorname{grad} f)(x), E_{k}(x)\rangle_{x} \text{ for }  1\leq k,j \leq d .
$
From the smoothness of the Riemannian metric (\ref{eq:RiemannianMetric}), it follows that $ x \mapsto \hat{\operatorname{Hess}}f(x) $ is a continuous function from $\mathcal{U} \subset \mathcal{M}$ to $\operatorname{Sym}(d)$. Since matrix norms are continuous, $\| \hat{\operatorname{Hess}}f(x)\|$ is continuous on $\mathcal{U} \ni \bar{x}.$ This argument holds for any $ \bar{x} \in \mathcal{M} $.
From Lemma \ref{lem:InvariantOrthonormalBasis}, $\| \hat{\operatorname{Hess}}f(x) \|_{2}=\| \operatorname{Hess}f(x) \|$ for any $x\in\mathcal{M} $, which completes the proof.
\end{proof}

The above result can be applied verbatim to the Hessian of constraint functions $\left\{ h_{i}\right\}_{i=1}^{l}$, 
$\left\{g_{i}\right\}_{i=1}^{m}$ in (\ref{RCOP}).
The next proposition can be proved similarly, as in Lemma \ref{lem:InvariantOrthonormalBasis} and Proposition \ref{prop:cont_Hessf}.

\begin{proposition}\label{cont_H&G}
	Consider $h,g$ in (\ref{RCOP}) and the linear operators $\mathcal{H}_{x}$, $ \mathcal{G}_{x}$ defined in (\ref{defn:HxGx}).
	Then,
	$x \mapsto \left\| \mathcal{H}_{x}\right\|$ and 
	$x \mapsto\left\|\mathcal{G}_{x}\right\|$
	are continuous scalar fields on $\mathcal{M}$.
\end{proposition}

\begin{proposition}\label{prop:matrix_nablaFw}
	Given $w=(x, y, z, s) \in \mathcal{N}$, consider the operator $\nabla F(w)$ in (\ref{eq:CovariantDerivativeKKTVectorField}). Let $\{E_{i}\}_{i=1}^{d}$ be an orthonormal basis of $ T_{x} \mathcal{M} $ and $ \{e_{i}\}_{i=1}^{l} $, $ \{\dot{e}_{i}\}_{i=1}^{m} $ be the standard bases of $ \mathbb{R}^{l}, \mathbb{R}^{m}$, respectively.
	If we choose an orthonormal basis of $T_{w} \mathcal{N}$ as follows:
	\begin{equation}\label{product_basis}
		\left\lbrace \left(E_{i}, 0,0,0\right) \right\rbrace_{i=1}^d \cup
		\left\lbrace \left(0_{x},e_{i}, 0,0\right) \right\rbrace_{i=1}^l \cup
		\left\lbrace \left(0_{x}, 0,\dot{e}_{i}, 0\right) \right\rbrace_{i=1}^m \cup
		\left\lbrace \left(0_{x}, 0,0,\dot{e}_{i}\right) \right\rbrace_{i=1}^m,
	\end{equation}
	then, the matrix representation of $\nabla F(w)$ is given by
	\begin{equation*}
	\hat{\nabla} F(w)=
	\left(\begin{array}{cccc}
		Q & B & C & 0 \\
		B^{T} & 0 & 0 & 0 \\
		C^{T} & 0 & 0 & I \\
		0 & 0 & S & Z
	\end{array}\right),
	\end{equation*}
	i.e., a matrix of order $ (d+l+2m) $ and, where 
	
	$Q:=Q(w) \in \operatorname{Sym}(d)$ is given by $ [Q]_{k j}:=\left\langle\operatorname{Hess}_{x} \mathcal{L}(w)[E_{j}], E_{k}\right\rangle_{x}$ for $1 \leq k, j \leq d$;
	
	$ B:=B(x)=[\hat{\operatorname{grad}} h_{1}(x), \cdots, \hat{\operatorname{grad}} h_{l}(x)] \in \mathbb{R}^{d \times l};$ 
	
	$ C:=C(x)=[ \hat{\operatorname{grad}} g_{1}(x), \cdots, \hat{\operatorname{grad}} g_{m}(x)] \in \mathbb{R}^{d \times m}$ and the ``hat'' symbol above means the corresponding vector representation under the basis $\{E_{i}\}_{i=1}^{d}$.
	
	In this case, there is a continuous scalar field $T\colon \mathcal{N} \to \mathbb{R}$ such that for any $w$, $ \|Q(w)\|_{\mathrm{F}} \leq T(w)$. 
	Moreover, $ x \mapsto\|B(x)\|_{\mathrm{F}}$ and $x \mapsto\|C(x)\|_{\mathrm{F}} $ are continuous scalar fields on $\mathcal{M}$.
\end{proposition}
\begin{proof}
The matrix $\hat{\nabla} F(w)$ under the basis (\ref{product_basis}) is obtained through a trivial process, so we will omit its description.
From relation (\ref{eq:HessianLagrange}),
we have
$
Q(w):=\hat{\operatorname{Hess}}_{x} \mathcal{L}(w)=
\hat{\operatorname{Hess}}f(x)
+\sum_{i=1}^{l} y_{i} \hat{\operatorname{Hess}} h_{i}(x)
+\sum_{i=1}^{m} z_{i} \hat{\operatorname{Hess}} g_{i}(x),
$
under the same basis. Thus,
$$
\left\| Q(w)\right\|_{\mathrm{F}}
\leq
\| \hat{\operatorname{Hess}}f(x)\|_{\mathrm{F}} 
+\sum_{i=1}^{l} |y_{i}| \| \hat{\operatorname{Hess}} h_{i}(x)\|_{\mathrm{F}}
+\sum_{i=1}^{m} |z_{i}| \| \hat{\operatorname{Hess}} g_{i}(x)\|_{\mathrm{F}}=:T(w).
$$
From Proposition \ref{prop:cont_Hessf}, $\|\hat{\operatorname{Hess}}f(x)\|_{\mathrm{F}}$,
$\|\hat{\operatorname{Hess}} h_{i}(x)\|_{\mathrm{F}}$,
and $\|\hat{\operatorname{Hess}} g_{i}(x)\|_{\mathrm{F}}$ are all continuous with respect to $ x $, thus, $ T $ is continuous.
As for $ x \mapsto\|B(x)\|_{\mathrm{F}}$, since the basis $\{E_{i}\}_{i=1}^{d}$ is orthonormal, 
$$
\|B(x)\|_{\mathrm{F}}^{2}
=\sum_{i=1}^{l}\|\hat{\operatorname{grad}} h_{i}(x)\|_{2}^{2}
=\sum_{i=1}^{l}\left\|\operatorname{grad} h_{i}(x)\right\|_{x}^{2},
$$
which implies continuity by (\ref{eq:RiemannianMetric}).
The claim for $ x \mapsto\|C(x)\|_{\mathrm{F}}$ can be proven similarly.
\end{proof}

\subsection{Global Convergence Theorem}

Now, we are ready to prove the global convergence Theorem \ref{thm:global} by following the procedure in \cite{el1996formulation}.
In what follows, we will omit similar content in \cite{el1996formulation} and focus on the difficulties encountered when adapting the proof of EIPM to RIPM. 

\begin{proposition}[Boundedness of the sequences]\label{boundedness}
	Let $\left\{w_{k}\right\}$ be a sequence generated by Algorithm \ref{algo:GlobalRIPM} and suppose that \ref{C1}-\ref{C3} hold.
	If $\epsilon>0$ and $w_{k} \in \Omega(\epsilon)$ for all $ k $, then
	\begin{enumerate}[(a)]
		\item $ \{z_{k}^{T} s_{k}\} $, $\left\{(z_{k})_{i}(s_{k})_{i}\right\}, i=1, \ldots, m$, are all bounded above and below away from zero;
		\item $\left\{z_{k}\right\}$ and $\left\{s_{k}\right\}$ are bounded above and component-wise bounded away from zero;
		\item $\left\{w_{k}\right\}$ is bounded;
		\item $\{ \|\nabla F (w_k)^{-1}\| \} $ is bounded;
		\item $\left\{\Delta w_{k}\right\}$ is bounded.
	\end{enumerate}
\end{proposition}
\begin{proof}
The proofs in \cite[Lemma 6.1]{el1996formulation} and/or \cite[Theorem 2 (a)]{bonettini2005inexact} can be applied verbatim to \textbf{(a)}, \textbf{(b)} and \textbf{(e)}.

\textbf{(c)} 
By (b), it suffices to prove that $\left\{y_{k}\right\}$ is bounded. The iteration count $k$ is omitted in what follows. 
By using the notation $ \mathcal{H}_{x}$ and $ \mathcal{G}_{x} $ as defined in (\ref{defn:HxGx}), we have
$
\mathcal{H}_{x}y
=
\operatorname{grad}_{x} \mathcal{L}(w)
-\operatorname{grad} f(x)
-\mathcal{G}_{x}z
=:b.
$
By \ref{C1}, $\mathcal{H}_{x}$ is an injection; then, there exists a unique solution $y$ to $\mathcal{H}_{x}y=b$. Indeed, we have
\begin{equation}\label{68} 
	y
	=[ \left( \mathcal{H}_{x}^{*} \mathcal{H}_{x}\right)^{-1} \mathcal{H}_{x}^{*}]\left( \operatorname{grad}_{x} \mathcal{L}(w)
	-\operatorname{grad} f(x)
	-\mathcal{G}_{x}z
	\right) .
\end{equation}
Define $\mathcal{C}_{x}: T_{x} \mathcal{M} \to \mathbb{R}^{l}$ as $ \mathcal{C}_{x}:= \left(\mathcal{H}_{x}^{*} \mathcal{H}_{x}\right)^{-1} \mathcal{H}_{x}^{*}.$
Under an arbitrary orthonormal basis of $T_{x} \mathcal{M}$ and standard basis of $\mathbb{R}^{l}$, if $ \hat{\mathcal{H}}_{x} $
is the matrix corresponding to $ \mathcal{H}_{x} $, then
$\hat{\mathcal{C}}_{x}=(\hat{\mathcal{H}}_{x}^{T} \hat{\mathcal{H}}_{x})^{-1} \hat{\mathcal{H}}_{x}^{T}$. 
It is easy to show that $\left\|\mathcal{C}_{x}\right\|=\|\hat{\mathcal{C}}_{x}\|_{2}$ for any $x$.
We can see that for each $ \bar{x} \in \mathcal{M}$ there is a neighborhood $\mathcal{U}$ of $ \bar{x} $ such that $ x \mapsto \hat{\mathcal{H}}_{x} $ is continuous over $\mathcal{U}$.
Then, by function composition,
$x \mapsto \hat{\mathcal{C}}_{x}$ is also continuous over $\mathcal{U}$.
This implies that $x \mapsto  \left\|\mathcal{C}_{x}\right\|=\|\hat{\mathcal{C}}_{x}\|_{2}$ is continuous at each $ \bar{x} $, and hence, on $ \mathcal{M}$.
Finally, by Proposition \ref{cont_H&G}, $ \left\|\mathcal{C}_{x}\right\|, \left\|\operatorname{grad} f(x)\right\|,$ and $\left\|\mathcal{G}_{x}\right\| $ are all continuous on $ \mathcal{M} $. Because $\{x_{k}\}$ is bounded, by (\ref{68}) we have
$
\left\|y_{k}\right\|
\leq
\left\| \mathcal{C}_{x_{k}} \right\|
\left(
\left\| \operatorname{grad}_{x} \mathcal{L}(w_{k})\right\|
+
\left\|\operatorname{grad} f(x_{k})\right\|
+ 
\left\|\mathcal{G}_{x_{k}}\right\| \left\| z_{k}\right\| 
\right)
\leq
c_{1}
\left(
\sqrt{\varphi_{0}}+
c_{2}
+ c_{3} \left\| z_{k}\right\| 
\right),
$
for some positive constants $ c_{1},c_{2},c_{3}.$
Then, $\left\{y_{k}\right\}$ is bounded because $\left\{z_{k}\right\}$ is bounded.

\textbf{(d)}
For each $ w_{k} $, choose an arbitrary orthonormal basis of $T_{w_{k}} \mathcal{N}.$
If the matrix representation $\hat{\nabla} F(w_{k})$ corresponds to $\nabla F(w_{k})$, then
$[\hat{\nabla} F(w_{k})]^{-1}$ corresponds to $\nabla F(w_{k})^{-1}$. 
By Lemma \ref{lem:InvariantOrthonormalBasis}, we have
$
\left\|\nabla F(w_{k})^{-1}\right\| \leq
\|[\hat{\nabla} F(w_{k})]^{-1}\|_{\mathrm{F}};
$
thus, it is sufficient to show that $\{\|[\hat{\nabla} F(w_{k})]^{-1}\|_{\mathrm{F}} \}$ is bounded.
For convenience, we will choose the basis of $T_{w_{k}} \mathcal{N}$ given in (\ref{product_basis}). 
Then, we have
\begin{equation*}
\hat{\nabla} F(w_{k})=
\left(\begin{array}{cccc}
	Q_{k} & B_{k} & C_{k} & 0 \\
	B^{T}_{k} & 0 & 0 & 0 \\
	C^{T}_{k} & 0 & 0 & I \\
	0 & 0 & S_{k} & Z_{k}
\end{array}\right).
\end{equation*}
By Proposition \ref{prop:matrix_nablaFw},
there is a continuous scalar field $T: \mathcal{N} \rightarrow \mathbb{R}$
such that $\|Q(w)\|_{\mathrm{F}} \leq T(w)$ for all $w \in \mathcal{N}$; and $\|B(x)\|_{\mathrm{F}}$, $\|C(x)\|_{\mathrm{F}}$ are continuous on $\mathcal{M}$.
It follows from the boundedness of $\left\{x_{k}\right\}$ and $\left\{w_{k}\right\}$ that for all $ k $,
$
\left\|Q_{k}\right\|_{\mathrm{F}} \equiv \left\|Q(w_{k})\right\|_{\mathrm{F}} \leq T(w_{k}) \leq c_{4},$ 
$\left\|B_{k}\right\|_{\mathrm{F}} \equiv \left\|B(x_{k})\right\|_{\mathrm{F}} \leq c_{5},$
and $\left\|C_{k}\right\|_{\mathrm{F}} \equiv \left\|C(x_{k})\right\|_{\mathrm{F}} \leq c_{6},$
for some positive constants $ c_{4}, c_{5},$ and $c_{6}.$

On the other hand, whichever basis is used in the form of (\ref{product_basis}), the structure of $\hat{\nabla} F(w_{k})$ and the properties of its block submatrices remain unchanged, e.g., symmetry of $ Q_{k} $; full rank of $B_{k}$; identity matrix $I$ in the third row; all zero matrices; diagonal matrices $S_{k} , Z_{k}$; etc. This ensures that we can obtain the desired result by performing an appropriate decomposition of $\hat{\nabla} F(w_{k})$. Up to this point, we have created all the conditions needed in the proof of the Euclidean version, namely, EIPM. We can make the claim that $\{\|[\hat{\nabla} F(w_{k})]^{-1}\|_{\mathrm{F}} \}$ is bounded by applying the proofs in \cite[Lemma 6.2]{el1996formulation} and/or \cite[Theorem 2 (c)]{bonettini2005inexact} directly.
\end{proof}

\begin{lemma}[$\left\{\bar{\alpha}_{k}\right\}$ bounded away from zero]\label{lem:Step length}
	Let $\left\{w_{k}\right\}$ be generated by Algorithm \ref{algo:GlobalRIPM} with $\mathrm{R} = \operatorname{Exp}$ and let \ref{C1}-\ref{C3} hold.
	If $\epsilon>0$ and $w_{k} \in \Omega(\epsilon)$ for all $ k $, $\left\{\sigma_{k}\right\}$ is bounded away from zero; then, $\left\{\bar{\alpha}_{k}\right\}$ is bounded away from zero.
\end{lemma}
\begin{proof} 
Since $\bar{\alpha}_{k}=\min \{\alpha_{k}^{I}, \alpha_{k}^{II}\}$, it is sufficient to show that $\{\alpha_{k}^{I}\}$ and $\{\alpha_{k}^{II}\}$ are bounded away from zero. 
For $\alpha_{k}^{I}$, see \cite[Lemma 6.3]{el1996formulation} and/or \cite[Theorem 3.1]{durazzi2000newton}. The proofs in those references apply verbatim to the Riemannian case. On the other hand, for $\alpha_{k}^{II}$, we need to adapt the proofs in \cite{el1996formulation,durazzi2000newton}, since Lipschitz continuity on manifolds is more complicated, see Subsection \ref{subsec:LipschitzContinuity}.

Let us suppress the subscript $k$.
Recall that $ w(\alpha)=\bar{\operatorname{Exp}}_{w}(\alpha \Delta w). $ Fix $ \alpha \Delta w $ and let $\mathrm{P}_{\gamma}$ be the parallel transport along the geodesic $\gamma(t)=\bar{\operatorname{Exp}}_{w}\left(t \alpha \Delta w\right)$. By Lemma \ref{lem:FundamentalTheoremOfCalculusRiemannian} where $c_1=0$, we obtain
\begin{align*}
	\mathrm{P}_{\gamma}^{0 \to 1} [F(w(\alpha))]
	&= F(w) +\alpha \nabla F(w) [\Delta w] -\alpha \nabla F(w) [\Delta w] +\int_{0}^{1} \mathrm{P}_{\gamma}^{0 \to t} \nabla F(\gamma(t)) \mathrm{P}_{\gamma}^{t \to 0} [\alpha \Delta w] \mathrm{d} t \\
	&= F(w) 
	+
	\alpha \left( \sigma \rho \hat{e}-F(w)\right) 
	+
	\alpha \int_{0}^{1} \left( \mathrm{P}_{\gamma}^{0 \to t} \nabla F(\gamma(t)) \mathrm{P}_{\gamma}^{t \to 0} 
	-
	\nabla F(w)\right) [\Delta w] \mathrm{d} t \\
	&= (1-\alpha)F(w) +\alpha \sigma \rho \hat{e}
	+\alpha \int_{0}^{1} \left(\mathrm{P}_{\gamma}^{0 \to t} \nabla F(\gamma(t)) \mathrm{P}_{\gamma}^{t \to 0} -\nabla F(w)\right) [\Delta w] \mathrm{d} t.
\end{align*}
Taking the norm on both sides above gives
\begin{align*}
	\left\| F(w(\alpha))\right\| 
	&=\left\| \mathrm{P}_{\gamma}^{0 \to 1} [F(w(\alpha))]\right\|\text{ (since parallel transport (\ref{eq:ParallelTransport}) is isometric)}\\
	&\leq (1-\alpha) \left\| F(w) \right\|+ \alpha \sigma \rho \left\|\hat{e}\right\| + 
	\alpha \int_{0}^{1} \left\|  \mathrm{P}_{\gamma}^{0 \to t} \nabla F(\gamma(t)) \mathrm{P}_{\gamma}^{t \to 0} -\nabla F(w)\right\|  \left\| \Delta w \right\| \mathrm{d} t \\
	&\leq (1-\alpha) \left\| F(w) \right\|+ \alpha \sigma \rho \sqrt{m} 
	+ 
	\alpha \int_{0}^{1} \kappa_2 \|t\alpha \Delta w \|  \left\| \Delta w \right\| \mathrm{d} t \\
	&= (1-\alpha) \left\| F(w) \right\|+ \alpha \sigma \rho \sqrt{m} + 
	\alpha^{2} \| \Delta w\|^{2} \kappa /2 . \label{86}
\end{align*}
The rest of the proof is the same as \cite[Lemma 6.3]{el1996formulation} and/or \cite[Theorem 3.1]{durazzi2000newton}, so we omit it.
\end{proof}


\vspace{2mm}
\begin{proof}[Proof of Theorem \ref{thm:global}:]
By Proposition \ref{prop:monotone}, we know that $\left\{\left\|F(w_{k})\right\|\right\}$ is monotonically nonincreasing, hence convergent. Assume that $\left\{\left\|F(w_{k})\right\|\right\}$ does not converge to zero.
Then, there exists $\epsilon>0$ such that
$\left\{w_{k} \right\} \subset \Omega(\epsilon)$ for infinitely many $k$.
We will show that the following two cases both lead to contradictions and thus, the hypothesis $\left\| F(w_{k}) \right\| \nrightarrow 0$ is not valid.

\textbf{Case 1.} 
For infinitely many $k$, if step (3b) in Algorithm \ref{algo:GlobalRIPM} is executed with $\alpha_{k} \equiv \bar{\alpha}_{k}$, it follows from Proposition \ref{prop:monotone} that
$
\varphi(w_{k+1}) / \varphi(w_{k}) \leq \lambda_{k}:= \left[1-2 \bar{\alpha}_{k} \beta\left(1-\sigma_{k}\right)\right].
$
Since $\left\{\bar{\alpha}_{k}\right\}$ is bounded away from zero by Lemma \ref{lem:Step length} and $\left\{\sigma_{k}\right\}$ is bounded away from one, then $\left\{\lambda_{k}\right\}$ is bounded away from one and hence, $\varphi(w_{k}) \to 0$; this is a contradiction.

\textbf{Case 2.} 
On the other hand, for infinitely many $k$, if $\alpha_{k}<\bar{\alpha}_{k}$, we have that $\alpha_{k} \leq \theta \bar{\alpha}_{k}$. Then, condition (\ref{eq:ArmijoCondition}) fails to hold for an $\tilde{\alpha}_{k}$ with $\alpha_{k}<\tilde{\alpha}_{k} \leq \alpha_{k} /\theta = \theta^{t-1} \bar{\alpha}_{k}$. Notice that $ \alpha_{k} /\theta $ is the value corresponding to the last failure.
Recall that the derivative of the real-valued function $\alpha \mapsto \varphi(\alpha):=\varphi\left(\bar{\operatorname{Exp}}_{w_{k}}(\alpha \Delta w_{k})\right)$ is
\begin{equation}\label{eq:241}
	\varphi^{\prime}(\alpha)=
	\mathcal{D} 
	\varphi (\bar{\operatorname{Exp}}_{w_{k}}\left(\alpha \Delta w_{k}\right))
	[\mathcal{D} \bar{\operatorname{Exp}}_{w_{k}}
	\left(\alpha \Delta w_{k}\right)\left[\Delta w_{k}\right]].
\end{equation}
Applying the mean value theorem to $ \varphi(\alpha) $ on the interval $ [0,\tilde{\alpha}_{k}] $ yields a number
$
\xi \in(0,1)$ such that $\tilde{\alpha}_{k} \varphi^{\prime}(\xi \tilde{\alpha}_{k})
=\varphi(\tilde{\alpha}_{k})-\varphi(0) .
$
For short, let $u:=\xi \tilde{\alpha}_{k} \Delta w_{k}.$
Hence,
\begin{equation}\label{eq:74}
	\begin{aligned}
		& \tilde{\alpha}_{k} \beta\left\langle\operatorname{grad} \varphi_{k}, \Delta w_{k}\right\rangle \\
		& <
		\varphi(\tilde{\alpha}_{k})-\varphi(0)  \text{ (as condition (\ref{eq:ArmijoCondition}) fails for $ \tilde{\alpha}_{k} $)} \\
		& =
		\tilde{\alpha}_{k} \varphi^{\prime}(\xi \tilde{\alpha}_{k}) \\ 
		&=
		\tilde{\alpha}_{k}
		\mathcal{D} 
		\varphi (\bar{\operatorname{Exp}}_{w_{k}}(u))
		\left[\mathcal{D} \bar{\operatorname{Exp}}_{w_{k}}
		(u)\left[\Delta w_{k}\right]\right] \text{ (by equation (\ref{eq:241}))}  \\
		&= 
		\tilde{\alpha}_{k}
		\langle 
		\operatorname{grad}
		\varphi (\bar{\operatorname{Exp}}_{w_{k}}(u)),
		\mathcal{D} \bar{\operatorname{Exp}}_{w_{k}}
		(u)\left[\Delta w_{k}\right]
		\rangle.
	\end{aligned}
\end{equation}
On the other hand, note that
\begin{equation}\label{78}
	\begin{aligned}
		&\left\langle\operatorname{grad} \varphi_{k}, \Delta w_{k}\right\rangle \\
		&=
		\langle \operatorname{grad} \varphi_{k}, u\rangle / \xi \tilde{\alpha}_{k}  \\
		&=
		\langle \mathcal{D} \bar{\operatorname{Exp}}_{w_{k}}
		(u)\left[\operatorname{grad} \varphi_{k}\right], 
		\mathcal{D} \bar{\operatorname{Exp}}_{w_{k}}
		(u)\left[u\right]
		\rangle / \xi \tilde{\alpha}_{k} \text{ (by Lemma \ref{lem:gauss})} \\
		&=
		\langle \mathcal{D} \bar{\operatorname{Exp}}_{w_{k}}
		(u)\left[\operatorname{grad} \varphi_{k}\right], 
		\mathcal{D} \bar{\operatorname{Exp}}_{w_{k}}
		(u)\left[\Delta w_{k}\right]
		\rangle. 
	\end{aligned}
\end{equation}
Subtracting $\tilde{\alpha}_{k} \left\langle\operatorname{grad} \varphi_{k}, \Delta w_{k}\right\rangle $ from both sides of (\ref{eq:74}) and using equalities (\ref{78}) gives
\begin{align*}
	& \tilde{\alpha}_{k}(\beta-1) 
	\left\langle\operatorname{grad} \varphi_{k}, \Delta w_{k}\right\rangle \\
	&< \tilde{\alpha}_{k}\left[ 
	\langle 
	\operatorname{grad}
	\varphi (\bar{\operatorname{Exp}}_{w_{k}}(u)),
	\mathcal{D} \bar{\operatorname{Exp}}_{w_{k}}
	(u)\left[\Delta w_{k}\right]
	\rangle
	-\langle\operatorname{grad} \varphi_{k}, \Delta w_{k}\rangle \right] \\
	&= \tilde{\alpha}_{k} 
	\langle 
	\operatorname{grad}
	\varphi (\bar{\operatorname{Exp}}_{w_{k}}(u))-\mathcal{D} \bar{\operatorname{Exp}}_{w_{k}}
	(u)\left[\operatorname{grad} \varphi_{k}\right],
	\mathcal{D} \bar{\operatorname{Exp}}_{w_{k}}
	(u)\left[\Delta w_{k}\right]
	\rangle \\
	&\leq \tilde{\alpha}_{k} 
	\left\|\operatorname{grad}
	\varphi (\bar{\operatorname{Exp}}_{w_{k}}(u))-\mathcal{D} \bar{\operatorname{Exp}}_{w_{k}}
	(u)\left[\operatorname{grad} \varphi_{k}\right]\right\|
	\left\|	\mathcal{D} \bar{\operatorname{Exp}}_{w_{k}}
	(u)\left[\Delta w_{k}\right]\right\| \\
	&= \tilde{\alpha}_{k} 
	\left\|\operatorname{grad}
	\varphi (y)-\mathcal{D} \bar{\operatorname{Exp}}_{w_{k}}
	(u)\left[\operatorname{grad} \varphi (w_{k}) \right]\right\|
	\left\|	\mathcal{D} \bar{\operatorname{Exp}}_{w_{k}}
	(u)\left[\Delta w_{k}\right]\right\| \\ 
	& \qquad \qquad \text{ (by letting $y:=\bar{\operatorname{Exp}}_{w_{k}}(u)$)}\\
	&\leq \tilde{\alpha}_{k} 
	\kappa\left\|u\right\| \left\|\Delta w_{k}\right\| \text{ (by (\ref{eq:DifferentiatedRetraction}), Definition \ref{defn:LipschitzContinuouslyDifferentiable} and Lemma \ref{lem:gauss})} \\
	&=  \kappa \xi \tilde{\alpha}_{k}^{2}  \left\|\Delta w_{k}\right\|^{2}.
\end{align*}
Finally, we obtain
$
(\beta-1) 
\left\langle\operatorname{grad} \varphi_{k}, \Delta w_{k}\right\rangle/(\kappa \xi \left\|\Delta w_{k}\right\|^{2})
< \tilde{\alpha}_{k}.
$

Because $ \left\langle\operatorname{grad} \varphi_{k}, \Delta w_{k}\right\rangle <0 $ and $\alpha_{k}$ satisfies (\ref{eq:ArmijoCondition}), we have
\begin{align*}
	\varphi_{k}(0)-\varphi_{k}(\alpha_{k})
	& \geq -\alpha_{k} \beta\left\langle\operatorname{grad} \varphi_{k}, \Delta w_{k}\right\rangle \\
	& \geq -\theta \beta \tilde{\alpha}_{k} \left\langle\operatorname{grad} \varphi_{k}, \Delta w_{k}\right\rangle \\
	& \geq -\theta \beta \left\langle\operatorname{grad} \varphi_{k}, \Delta w_{k}\right\rangle (\beta-1) 
	\left\langle\operatorname{grad} \varphi_{k}, \Delta w_{k}\right\rangle/(\kappa \xi \left\|\Delta w_{k}\right\|^{2})\\
	&\geq [\theta \beta (1-\beta) /\kappa \xi]
	\left( \left\langle\operatorname{grad} \varphi_{k}, \Delta w_{k}\right\rangle / \left\|\Delta w_{k}\right\|\right)^{2} \\
	&=\omega
	\left( \left\langle\operatorname{grad} \varphi_{k}, \Delta w_{k}\right\rangle / \left\|\Delta w_{k}\right\|\right),
\end{align*}
where $\omega(\cdot)$ is an $F$-function (see \cite[Definition 14.2.1 \& 14.2.2 in P479]{ortega1970iterative}).
Since $\{\varphi_{k}\}$ is bounded below and $\varphi_{k} \geq \varphi_{k+1}$, it follows that $\lim _{k \to \infty}(\varphi_{k}-\varphi_{k+1})=0$. By the definition of $ F $-functions, we obtain
$
\left\langle\operatorname{grad} \varphi_{k}, \Delta w_{k}\right\rangle / \left\|\Delta w_{k}\right\| \to 0.
$
Since $\{\left\|\Delta w_{k}\right\|\}$ is bounded (Proposition \ref{boundedness}), we have
$
\left\langle\operatorname{grad} \varphi_{k}, \Delta w_{k}\right\rangle \to 0.
$
Choosing $ \rho_{k} $ with $z_{k}^{T} s_{k}/m \leq \rho_{k} \leq \left\|F(w_{k})\right\| / \sqrt{m}$ in Algorithm \ref{algo:GlobalRIPM} implies that
\begin{align*}
	\left\langle\operatorname{grad} \varphi_{k}, \Delta w_{k}\right\rangle /(-2)
	&=
	\varphi_{k}-\sigma_{k} \rho_{k} z_{k}^{T} s_{k}
	\geq \varphi_{k}-\sigma_{k}\left\|F(w_{k})\right\| z_{k}^{T} s_{k} / \sqrt{m} \\
	&\geq  \varphi_{k}-\sigma_{k}\left\|F(w_{k})\right\|^{2}
	\geq  (1-\sigma_{k})\varphi_{k}.
\end{align*}
This shows that $\varphi (w_{k}) \to 0$, because $\left\{\sigma_{k}\right\}$ is bounded away from one; this is a contradiction.
\end{proof}

\section{Numerical Experiments}\label{sec:ne}

The numerical experiments compared the performance of the globally convergent RIPM (Algorithm \ref{algo:GlobalRIPM}) with those of other Riemannian methods in solving two problems. They were conducted in Matlab R2022a on a computer with an Intel Core i7-10700 (2.90GHz) and 16GB RAM. Algorithm \ref{algo:GlobalRIPM} utilized Manopt 7.0 \cite{Boumal2014}, a Riemannian optimization toolbox for Matlab. 
The problems involve three manifolds:
\begin{itemize}
	\item fixed-rank manifold, $\mathcal{M}_r =\left\{X \in \mathbb{R}^{m \times n}: \operatorname{rank}(X)=r\right\}$; 
	\item Stiefel manifold, $\operatorname{St}(n, k) =\left\{X \in \mathbb{R}^{n \times k}: X^T X=I_k\right\}$;
	\item oblique manifold, $\mathrm{Ob}(n, k) =\{X \in \mathbb{R}^{n \times k}: \left(X^T X\right)_{i i}=1, \forall i=1,\dots,k \}.$
\end{itemize}
We only consider their embedded geometry
and we apply the default retractions in Manopt, e.g., the retraction based on QR decomposition for the Stiefel manifold.

\textbf{Problem I.} 
Recently, \cite{Song2020} proposed the Nonnegative Low-Rank Matrix (NLRM) approximation. Formally, NLRM aims to solve
\begin{equation}\label{pro:nlrm} \tag{NLRM}
	\min _{X \in \mathcal{M}_r}\|A-X\|_{\mathrm{F}}^2 \quad \text { s.t. } X \geq 0.
\end{equation}
\emph{Input.} 
We tested three cases of integer triples $(m, n, r)$: $(20, 16, 2)$, $(30, 24, 3)$, and $(40, 32, 4)$.
For each $(m,n,r)$, we generated nonnegative $L \in \mathbb{R}^{m \times r}$, $R \in \mathbb{R}^{r \times n}$ whose entries follow a uniform distribution in [0,1]. Then, we added the Gaussian noise with zero mean and different standard deviation $(\sigma = 0, 0.001, 0.01)$ to $ A=LR$.  
When there is no noise (i.e., $\sigma = 0$), the input data matrix $A$ itself is exactly a solution.

\textbf{Problem II.} 
Given $C \in \mathbb{R}^{n \times k}$, \cite{Jiang2022} computed its projection onto the nonnegative part of the Stiefel manifold. If the distance is measured in terms of $\|C-X\|_{\mathrm{F}}^2$, we can express it equivalently as
\begin{equation}\label{pro:nonStProj} \tag{Model\_St}
	\min _{X \in \operatorname{St}(n, k)} - 2 \operatorname{trace}(X^{T}C)  \quad \text { s.t. } X \geq 0.
\end{equation}
In \cite{Jiang2022}, it is shown that (\ref{pro:nonStProj}) can be equivalently reformulated into
\begin{equation}\label{pro:nonStProj_ob} \tag{Model\_Ob}
	\min _{X \in \operatorname{Ob}(n, k)} - 2 \operatorname{trace}(X^{T}C)  \quad \text { s.t. } X \geq 0, \; \|X V\|_{\mathrm{F}}=1,
\end{equation}
where the positive integer $p$ and $V \in \mathbb{R}^{k \times p}$ can be arbitrary as long as $\|V\|_{\mathrm{F}}=1$ and $V V^T$ is entry-wise positive. We examined both models. 
\emph{Input.} 
We considered four cases of integer pairs $(n,k)$: $(40,8),(50,10),(60,16),$ and $(70,14).$
For a general $C$, it is always difficult to seek nonnegative projections globally. Fortunately, Proposition 1 in \cite{Jiang2022} showed a way to construct $ C $ such that \ref{pro:nonStProj} has a unique, known solution $ X^{*} $. In our experiments, we generated a feasible point $ B $ of \ref{pro:nonStProj}; then, we set $C$ by codes:
\texttt{X1=(B>0).*(1+rand(n,k));Xstar=X1./sqrt(sum(X1.*X1));L=rand(k,k);\\L=L+k*eye(k);C=Xstar*L'.}
The initial point was computed by projecting $ C $ onto the Stiefel manifold. 
In addition, for \ref{pro:nonStProj_ob} we set \texttt{p=1;V=ones(k,p);V=V/norm(V,"fro").}

\subsection{Implementation Details}

The experimental implementation of Algorithm \ref{algo:GlobalRIPM} (i.e.,RIPM) initialized $z_0$ and $s_0$ from a uniform distribution in [0,1] and set $y_0=0$ if equality constraints exist. 
We used $\rho_{k}=z_{k}^{T}s_{k}/m$, $\sigma_k=\min \{0.5, \| F(w_{k})\| ^{1/2}\}$ and Algorithm \ref{algo:ConjugateResidualTangentSpace} to solve the condensed form of Newton equation (\ref{eq:CondensedEquation}). Algorithm \ref{algo:ConjugateResidualTangentSpace} stopped when the relative residual went below $10^{-9}$, or it reached 1000 iterations.
We used a backtracking line search simultaneously for the central conditions and sufficient decreasing conditions. 
We set $\gamma_{-1}=0.9$, $\gamma_{k+1}=(\gamma_{k}+0.5)/2;$ and $\beta=10^{-4}$, $\theta=0.5$.
We compared RIPM with the following Riemannian methods:
\begin{itemize}
	\item RALM: Riemannian augmented Lagrangian method \cite{liu2020simple}.
	\item REPM$_{\text{lqh}}$: Riemannian exact penalty method with smoothing functions of linear-quadratic and pseudo-Huber \cite{liu2020simple}.
	\item REPM$_{\text{lse}}$: Riemannian exact penalty method with smoothing functions of log-sum-exp \cite{liu2020simple}.
	\item RSQP: Riemannian sequential quadratic programming \cite{Obara2022}.
\end{itemize}
Let $[t]_{+}:=\max(0,t)$ and $ [t]_{-}:=\min(0,t)$ for any $t \in \mathbb{R}.$
The experimental settings followed those of \cite{Obara2022}, where they used the KKT residual defined as
\begin{equation*}
	\sqrt{ \left\| \operatorname{grad}_{x}\mathcal{L}(w) \right\|^2 +\sum_{i =1}^{m} \{  [z_{i}]_{-}^2 + [g_i(x)]_{+}^2 +\left|z_i g_i(x)\right|^2\} +\sum_{i =1}^{l} \left|h_i(x)\right|^2 }
\end{equation*}
to measure the deviation of an iterate from the KKT conditions.
For the parameters of RALM, REPMs, and RSQP, we utilized the experimental setting and Matlab codes provided by \cite{Obara2022}.

We conducted 20 random trials of each problem and model. 
All the algorithms ran with the same initial point.
The experiment is considered successfully terminated if it finds a solution with a KKT residual lower than $\epsilon_{\text{kkt}}$ before triggering any of the stopping conditions.
For the \emph{first-order} algorithms (including RALM and the REPMs), the stopping conditions are: elapsed time exceeding $ t_{\max} $ seconds, number of outer iterations exceeding 1,000, or failure of the algorithm to update any parameters.
For the \emph{second-order} algorithms (including RSQP and RIPM), the stopping conditions are elapsed time exceeding $ t_{\max} $ seconds or a number of outer iterations exceeding 10,000.
Here, considering that some problems might not have converged easily, the maximum number of iterations was chosen to be 1,000 (10,000), which was a sufficiently large value.
The selection of $t_{\max}$ is related to the actual time it took to run all the codes on the computer. Setting $t_{\max}$ too large resulted in excessive time spent on poorly performing algorithms. 
On the other hand, $\epsilon_{\mathrm{kkt}}$ was chosen to better demonstrate that second-order algorithms could achieve more accurate solutions. 
Therefore, we chose the appropriate values for $ t_{\max} $ and $\epsilon_{\text{kkt}}$ according to the problem that was to be solved.

\subsection{Results and Analysis}

The tables in this subsection report the success rate (\textbf{Success}), the average time (\textbf{Time}), and the average iteration number (\textbf{Iter.}) among the successful trials.
In order to capture the combination of stability and speed, the algorithms with 0.9 or higher are first highlighted in bold in the ``Success'' column, and then, of those algorithms, the fastest result is highlighted in bold in the ``Time'' column.
Here, ``successful convergence in numerical experiments'' means that the algorithm can generate a relatively accurate solution in a relatively reasonable amount of time. It is not exactly the same as the ``theoretical convergence in any global convergence theorem''. 
Numerical experiments reflect the actual performance of the algorithm in the application. 
For example, the first-order algorithms (RALM, REPMs) also have theoretical global convergence, but under our high criterion, it is difficult for them to generate a high-precision solution within a certain period of time.
The second-order algorithm (RIPM, RSQP), on the other hand, is excellent in terms of accuracy, although it takes quite a bit of time.

\textbf{Problem I.} 
Here, we set $ t_{\max} = 180$, $\epsilon_{\text{kkt}} = 10^{-8}$.
The numerical results in Table \ref{table:nlrm} show that RIPM had the best performance, except for cases (30, 40, 3) and (40, 32, 4) without noise.
The time spent by RALM and the REPMs grew slowly with the problem size $m$ and $n$, but their success rates dropped sharply as the noise level (standard deviation $\sigma$) intensified, eventually leading to non-convergence.
In contrast, RSQP and RIPM were more stable and robust, while RIPM was much faster than RSQP.
The cost of RSQP increased drastically with the problem size because a \emph{quadratic programming subproblem} (defined over tangent space of current point) had to be solved in each iteration. 
Unlike RIPM, which uses Krylov subspace methods introduced in Subsection \ref{ssec:Krylov} to avoid expensive computations, RSQP had to transform the subproblem into a matrix representation (similar to Step 1-6 in Subsection \ref{ssec:Krylov}).
As can be seen from Table \ref{table:nlrm}, RIPM took about the same amount of time as RALM and the REPMs did.

\begin{table}
	\caption{Performance of various Riemannian methods on problem I (\ref{pro:nlrm}).}
	\centering
	\label{table:nlrm}
	\begin{tabular}{l|lll|lll|lll}
		\hline
		$(m,n,r)$ & \multicolumn{3}{c}{(20,16,2)} & \multicolumn{3}{c}{(30,24,3)} & \multicolumn{3}{c}{(40,32,4)} \\ 
		\hline
		no noise  & Success & Time [s] & Iter. & Success & Time [s] & Iter. & Success & Time [s] & Iter. \\ 
		\hline
		RALM                           & 0.4           & 1.115                          & 31    & 0.65          & {1.813}                		& 31    & 0.75          & {2.800}                		& 31    \\
		REPM$_{\text{lqh}}$            & \textbf{1}    & {5.165$\times 10^{-1}$} 		& 31    & \textbf{1}    & \textbf{1.009}                & 31    & \textbf{1}    & \textbf{1.747}                & 31    \\
		REPM$_{\text{lse}}$            & \textbf{1}    & 2.242                          & 31    & \textbf{1}    & 4.041                         & 31    & \textbf{0.95} & 6.952                         & 31    \\
		RSQP                           & \textbf{0.9}  & 6.429                          & 7     & \textbf{0.9}  & 3.944$\times 10$              & 8     & \textbf{0.9}  & 1.254$\times 10^{2}$          & 8     \\
		RIPM                           & \textbf{1}    & $\mathbf{4.920\times 10^{-1}}$ & 19    & \textbf{1}    & 2.247                         & 27    & \textbf{1}    & 5.277                         & 32    \\ 
		\hline \hline
		$(m,n,r)$ & \multicolumn{3}{c}{(20,16,2)} & \multicolumn{3}{c}{(30,24,3)} & \multicolumn{3}{c}{(40,32,4)} \\ 
		\hline
		$\sigma=0.001$ & Success & Time [s] & Iter. & Success & Time [s] & Iter. & Success & Time [s] & Iter. \\ 
		\hline
		RALM                           & 0.2           & 1.001                          & 31    & 0.15          & 2.050                         & 31    & 0.05          & 2.758                         & 31    \\
		REPM$_{\text{lqh}}$            & 0.1           & 4.983$\times 10^{-1}$          & 32    & 0.25          & 1.035                         & 31    & 0.15          & 1.787                         & 31    \\
		REPM$_{\text{lse}}$            & 0.15          & 2.444                          & 31    & 0.1           & 4.867                         & 31    & 0.05          & 8.371                         & 31    \\
		RSQP                           & \textbf{0.95} & {6.619}                 		& 7     & \textbf{0.95} & {3.848$\times 10$}     		& 8     & \textbf{0.9}  & 1.299$\times 10^{2}$          & 8     \\
		RIPM                           & \textbf{1}    & $\mathbf{5.376\times 10^{-1}}$ & 20    & \textbf{1}    & \textbf{2.342}                & 27    & \textbf{1}    & \textbf{4.631}                & 29    \\ 
		\hline \hline
		$(m,n,r)$ & \multicolumn{3}{c}{(20,16,2)} & \multicolumn{3}{c}{(30,24,3)} & \multicolumn{3}{c}{(40,32,4)} \\ 
		\hline
		$\sigma=0.01$ & Success & Time [s] & Iter. & Success & Time [s] & Iter. & Success & Time [s] & Iter. \\ 
		\hline
		RALM                           & 0             & -                              & -     & 0             & -                             & -     & 0             & -                             & -     \\
		REPM$_{\text{lqh}}$            & 0             & -                              & -     & 0             & -                             & -     & 0             & -                             & -     \\
		REPM$_{\text{lse}}$            & 0             & -                              & -     & 0             & -                             & -     & 0             & -                             & -     \\
		RSQP                           & \textbf{1}    & {7.295}                 		& 8     & \textbf{0.95} & {4.114$\times 10$}            & 8     & \textbf{0.95} & {1.430$\times 10^{2}$} 		& 9     \\
		RIPM                           & \textbf{1}    & \textbf{5.980$\times 10^{-1}$} & 21    & \textbf{0.95} & \textbf{1.883}                & 25    & \textbf{0.95} & \textbf{4.602}                & 29    \\ 
		\hline
	\end{tabular}
\end{table}

\textbf{Problem II.}
Here, we set $ t_{\max} = 600$, $\epsilon_{\text{kkt}} = 10^{-6}$ for both (\ref{pro:nonStProj}) and (\ref{pro:nonStProj_ob}). Since the true solution $X^{*}$ is known, we added a column showing the average error $\|\tilde{X} - X^{*} \|_{\mathrm{F}}$, where $ \tilde{X} $ denotes the final iterate.
The numerical results are listed in Table \ref{table:nonStProj} and \ref{table:nonStProj_ob}. The Error columns show that if the KKT residual is sufficiently small, then $ \tilde{X} $ does approximate the true solution. In particular, RSQP and RIPM yield a more accurate solution (the error is less than $10^{-7} $). From Table \ref{table:nonStProj}, we can see that RALM is stable and fast for (\ref{pro:nonStProj}). However, from Table \ref{table:nonStProj_ob}, the success rate of RALM for (\ref{pro:nonStProj_ob}) decreases as the problem size becomes larger. The REPMs do not work at all on either model. 
RSQP also does not perform well on either model.
In contrast, RIPM successfully solved all instances of both models and was only slightly slower than RALM in some cases. Overall, our RIPM was relatively fast and most stable.

\begin{table}
	\caption{Performance of various Riemannian methods on problem II (\ref{pro:nonStProj}).}
	\centering
	\label{table:nonStProj}
	\begin{tabular}{l|llll|llll}
		\hline
		$(n,k)$   & \multicolumn{4}{c}{(40,8)}          & \multicolumn{4}{c}{(50,10)}         \\ \hline
		& Success & Time [s]  & Iter. & Error    & Success & Time [s]  & Iter. & Error    \\ \hline
		RALM      & \textbf{1}    & {2.347}     & 45    & 5.41$\times 10^{-7}$ & \textbf{1}    & {4.344}     & 54    & 5.21$\times 10^{-7}$ \\ 
		REPM$_{\text{lqh}}$ & 0    & -         & -     & -        & 0    & -         & -     & -        \\
		REPM$_{\text{lse}}$ & 0    & -         & -     & -        & 0    & -         & -     & -        \\
		RSQP      & \textbf{0.9}  & 1.352$\times 10$ & 7     & 2.05$\times 10^{-9}$ & 0.7  & 3.097$\times 10$ & 6     & 2.47$\times 10^{-9}$ \\
		RIPM      & \textbf{1}    & \textbf{2.225}     & 31    & 3.72$\times 10^{-8}$ & \textbf{1}    & \textbf{3.785}     & 32    & 3.38$\times 10^{-8}$ \\ \hline \hline
		$(n,k)$   & \multicolumn{4}{c}{(60,12)}         & \multicolumn{4}{c}{(70,14)}         \\ \hline
		& Success & Time [s]  & Iter. & Error    & Success & Time [s]  & Iter. & Error    \\ \hline
		RALM      & \textbf{1}    & \textbf{4.097}     & 34    & 4.93$\times 10^{-7}$ & \textbf{1}    & \textbf{6.234}     & 37    & 5.34$\times 10^{-7}$ \\
		REPM$_{\text{lqh}}$ & 0    & -         & -     & -        & 0    & -         & -     & -        \\
		REPM$_{\text{lse}}$ & 0    & -         & -     & -        & 0    & -         & -     & -        \\
		RSQP      & 0.65 & 7.802$\times 10$ & 7     & 6.48$\times 10^{-9}$ & 0.85 & 1.661$\times 10^{2}$ & 7     & 2.64$\times 10^{-9}$ \\
		RIPM      & \textbf{1}    & {5.555}     & 32    & 2.81$\times 10^{-8}$ & \textbf{1}    & {7.574}     & 33    & 2.45$\times 10^{-8}$ \\ \hline
	\end{tabular}
\end{table}

\begin{table}
	\caption{Performance of various Riemannian methods on problem II (\ref{pro:nonStProj_ob}).}
	\centering
	\label{table:nonStProj_ob}
	\begin{tabular}{l|llll|llll}
		\hline
		$(n,k)$          & \multicolumn{4}{c}{(40,8)}                         & \multicolumn{4}{c}{(50,10)}                        \\ \hline
		& Success       & Time [s]           & Iter. & Error    & Success       & Time [s]           & Iter. & Error    \\ \hline
		RALM             & \textbf{1} & \textbf{2.510}     & 51    & 5.04$\times 10^{-7}$ & \textbf{0.95}       & \textbf{4.727}              & 64    & 4.94$\times 10^{-7}$ \\
		REPM$_{\text{lqh}}$        & 0          & -                  & -     & -        & 0          & -                  & -     & -        \\
		REPM$_{\text{lse}}$        & 0          & -                  & -     & -        & 0          & -                  & -     & -        \\
		RSQP             & 0.65       & 8.618              & 5     & 2.30$\times 10^{-10}$ & 0.7        & 2.782$\times 10$          & 6     & 1.12$\times 10^{-10}$ \\	
		RIPM         & \textbf{1} & {3.791}     & 22    & 5.62$\times 10^{-9}$ & \textbf{1} & {5.880}     & 23    & 7.93$\times 10^{-9}$ \\ 
		\hline \hline
		\textbf{$(n,k)$} & \multicolumn{4}{c}{(60,12)}                        & \multicolumn{4}{c}{(70,14)}                        \\ \hline
		& Success       & Time [s]           & Iter. & Error    & Success       & Time [s]           & Iter. & Error    \\ \hline
		RALM             & 0.6        & 5.725              & 49    & 3.82$\times 10^{-7}$ & 0.6        & 8.223              & 52    & 3.85$\times 10^{-7}$ \\
		REPM$_{\text{lqh}}$        & 0          & -                  & -     & -        & 0          & -                  & -     & -        \\
		REPM$_{\text{lse}}$        & 0          & -                  & -     & -        & 0          & -                  & -     & -        \\
		RSQP             & 0.7        & 4.446$\times 10$          & 5     & 1.17$\times 10^{-9}$ & 0.5        & 9.138$\times 10$          & 5     & 1.82$\times 10^{-9}$ \\	
		RIPM         & \textbf{1} & \textbf{7.134}     & 23    & 9.69$\times 10^{-9}$ & \textbf{1} & \textbf{9.268}     & 24    & 1.06$\times 10^{-8}$ \\ 
		\hline
	\end{tabular}
\end{table}

\section{Conclusions}\label{sec:Conclusion}

In this paper, we proposed a Riemannian version of the classical interior point method and established its local and global convergence. To our knowledge, this is the first study to apply the primal-dual interior point method to the nonconvex constrained optimization problem on a Riemannian manifold. 
Numerical experiments showed the stability and efficiency of our method.

Recently, Hirai et al. \cite{hirai2023interior} extended the self-concordance-based interior point methods to Riemannian
manifolds.
They aimed to minimize a \emph{geodesically convex} (i.e., convex on manifolds) objective $f \colon D \rightarrow \mathbb{R}$ defined on a geodesically convex subset $D \subset \mathcal{M}$.
In contrast, in (\ref{RCOP}), we do not require any convexity. 
In practice, many convex functions (in the Euclidean sense) are not geodesically convex on some interested manifolds.
For example, for any geodesically convex function defined on a connected, compact Riemannian manifold (e.g., Stiefel manifold), it must be constant \cite[Corollary 11.10]{boumal2022intromanifolds}, which is not of interest in the field of optimization.
Thus, (\ref{RCOP}) has a wider applicability.

In closing, let us make a comparison with the Euclidean Interior Point Method (EIPM) to illustrate the theoretical advantages of our RIPM and discuss two future directions of research on more advanced RIPM methods.
\paragraph{Comparison: Riemannian IPM (RIPM) v.s. Euclidean IPM (EIPM).}
\begin{enumerate}
	\item 
	RIPM generalizes EIPM from Euclidean space to general Riemannian manifolds. EIPM is a special case of RIPM when $\mathcal{M} = \mathbb{R}^n$ or $\mathbb{R}^{m \times n}$ in (\ref{RCOP}). 
	\item RIPM inherits all the advantages of Riemannian optimization.
	For example, we can exploit the geometric structure of $\mathcal{M}$, which is usually regarded as a set of constraints from the Euclidean viewpoint.
	\item Note that in both RIPM and EIPM, we have to solve the condensed Newton equation (\ref{eq:CondensedEquation}) at each iteration. 
	However, if the equality constraints can be considered to be a manifold, RIPM can solve (\ref{eq:CondensedEquation}) with a \emph{smaller} order on $T_x \mathcal{M} \times \mathbb{R}^l$.
	For example, the problem II (\ref{pro:nonStProj}) can be rewritten as
	\begin{equation*}
		\min _{X \in \mathbb{R}^{n \times k}} 
		- 2 \operatorname{trace}(X^{T}C)  \quad \text { s.t. } X^{T} X = I_k, \; X \geq 0.
	\end{equation*}
	Here, Stiefel manifold is replaced by the equality constraints, i.e., we define $h \colon \mathbb{R}^{n \times k} \to \operatorname{Sym}(k): X \mapsto h(X):=X^T X-I_k$; and $\mathcal{M} = \mathbb{R}^{n \times k}$, $ l = \operatorname{dim} \operatorname{Sym}(k) = k(k+1) / 2$ in (\ref{RCOP}).
	Then, when we apply EIPM, it requires us to solve (\ref{eq:CondensedEquation}) of order $n k + k(k+1) / 2$.
	On the other hand, if we apply RIPM to (\ref{pro:nonStProj}), then (\ref{eq:CondensedEquation}) reduces to (\ref{eq:OnlyInequality}) since there are only inequality constraints on $\mathcal{M} =\mathrm{St}(n, k)$.
	In this case, we solve the equation of order $n k-k(k+1) / 2$, i.e., the dimension of $\mathrm{St}(n, k)$.
	Compared to EIPM, using RIPM reduces our dimensionality by $k(k+1)$.
	\item RIPM can solve some problems that EIPM cannot. 
	For example, the problem I (\ref{pro:nlrm}) can be rewritten as
	\begin{equation*}
		\min _{X \in \mathbb{R}^{m \times n}} \|A-X\|_{\mathrm{F}}^2 \quad \text { s.t. } \operatorname{rank}(X)=r, \; X \geq 0.
	\end{equation*}
	Since the rank function, $X \mapsto \operatorname{rank}(X)$, is not even continuous, we cannot apply EIPM.
\end{enumerate}

\paragraph{Future Work I: Preconditioner for linear operator equation.}

With regard to the complementary condition, $S_{k}^{-1} Z_{k}$ values display a huge difference in magnitude as $k\to \infty$. 
The operator $\Theta:=G_x S^{-1} Z G_x^*$ causes the system (\ref{eq:CondensedEquation}) to be ill-conditioned, risking failure of the iterative method without preconditioning.
Matrix-decomposition-based preconditioner methods cannot be applied to a problem that does not have a matrix form. 
A possible alternative is to find a nonsingular $\mathcal{P}$ such that the condition number of $\mathcal{P}^{-1} \mathcal{T}$ is smaller.

\paragraph{Future Work II: Treatment of more state-of-the-art interior point methods.}

While we have considered interior point methods on a manifold for the first time, our Euclidean theoretic counterpart is an early nonlinear interior point method algorithm \cite{el1996formulation}; however, the counterpart now appears to be obsolete compared with more recent interior point methods.
For example, our method does not drive the iteration towards minimizers but only towards stationary points; globalization is done by monitoring only the KKT residuals; moreover, the boundedness assumption \ref{C2} of $\{z_{k}\}$ is too strong to hold in some simple cases (see W{\"a}chter-Biegler effect \cite{wachter2000failure}). 
It remains an important issue to adapt more modern interior point methods to manifolds, although we may encounter various difficulties in Riemannian geometry.

\smallskip

\noindent
\textbf{Acknowledgements}
We are deeply grateful to Hiroyuki Sato of Kyoto University for his valuable comments and suggestions. 
This work was supported by JST SPRING, Grant Number JPMJSP2124, and JSPS KAKENHI, Grant Numbers 19H02373 and 22K18866. 

\smallskip

\noindent
\textbf{Data availability}
The datasets generated during and/or analyzed during the current study are available from the corresponding author upon reasonable request.




\end{document}